\documentclass[11pt]{article}

\usepackage{amscd, amsfonts, amsmath, amssymb, amsthm}
\usepackage{appendix}
\usepackage{cite}
\usepackage{float}
\usepackage{geometry}
\usepackage{graphicx}
\usepackage{hyperref}
\usepackage{microtype}
\usepackage{xcolor}

\hypersetup{
    colorlinks = true,
    linkcolor  = blue,
    citecolor  = red,
    urlcolor   = magenta
}

\numberwithin{equation}{section}

\newtheorem{theorem}{theorem}[section]

\newtheorem{definition}[theorem]{Definition}
\newtheorem{example}[theorem]{Example}
\newtheorem{proposition}[theorem]{Proposition}

\begin{document}

\thispagestyle{empty}

\hfill \today

\vspace{0.5cm}

\begin{center}

\LARGE{\textbf{Universal $T$-matrices for quantum Poincaré groups: contractions and quantum reference frames}}

\end{center}

\begin{center}

\textsc{Angel Ballesteros}$^1$, \textsc{Diego Fernandez-Silvestre}$^2$, \textsc{Ivan Gutierrez-Sagredo}$^2$

\medskip

$^1$\textit{Departamento de Física, Universidad de Burgos, 09001 Burgos, Spain}

$^2$\textit{Departamento de Matemáticas y Computación, Universidad de Burgos, 09001 Burgos, Spain}

\medskip
 
Emails: \href{mailto:angelb@ubu.es}{angelb@ubu.es}, \href{mailto:dfsilvestre@ubu.es}{dfsilvestre@ubu.es}, \href{mailto:igsagredo@ubu.es}{igsagredo@ubu.es}

\end{center}

\begin{abstract}

Universal $T$-matrices, or Hopf algebra dual forms, for quantum groups are revisited, and their contraction theory is developed. As a first illustrative example, the (1+1) timelike $\kappa$-Poincaré $T$-matrix is explicitly worked out. Afterwards, motivated by recent results on the role of the Hopf algebra dual form of a quantum (1+1) centrally extended Galilei group as the algebraic object underlying non-relativistic quantum reference frame transformations, a new quantum deformation of the (1+1) centrally extended Poincaré Lie algebra is obtained, and its universal $T$-matrix is presented. Finally, the Hopf algebra dual form contraction is applied to this Poincaré $T$-matrix, showing that its corresponding non-relativistic counterpart is precisely the Galilei $T$-matrix associated with quantum reference frames. In this way, the Poincaré Hopf algebra dual form introduced here stands as a natural candidate for describing the symmetry structure of relativistic quantum reference frame transformations. In the appropriate basis, the associated quantum Poincaré group is recognized, remarkably, as a non-trivial central extension of the (1+1) spacelike $\kappa$-Poincaré dual Hopf algebra.

\end{abstract}


\medskip

\noindent \textsc{Keywords:}
Hopf algebras, universal $T$-matrices, quantum groups, Poincaré group, Galilei group, central extensions, contractions, quantum reference frames

\tableofcontents


\section{Introduction}

Let $H$ be a Hopf algebra. Its Hopf algebra dual form or universal $T$-matrix $T \in H^* \otimes H$ is defined as
\begin{equation} \label{T0}
    T := x_\mu \otimes X^\mu \ ,
\end{equation}
where the $X^\mu \in H$ are basis elements of the Hopf algebra $H$ and the $x_\mu \in H^*$ are basis elements of the dual Hopf algebra $H^*$ (summation over repeated indexes is assumed). 

In this paper, we will consider the Hopf algebra $H$ to be a quantum universal enveloping algebra $U_q(\mathfrak{g})$, where $\mathfrak{g}$ is the Lie algebra of a finite-dimensional Lie group $G$. Accordingly, the (restricted) dual Hopf algebra $H^*$ will be the quantum group, namely the noncommutative analogue of the commutative algebra of regular functions $\mathcal{O}(G)$ on the Lie group $G$ (with $G$ arising as the $\mathbb{K}$-points of a connected affine algebraic group defined over a field $\mathbb{K}$ of characteristic 0)\footnote{Note that, if $G$ is an affine algebraic group defined over a field $\mathbb{K}$ of characteristic 0 (typically $\mathbb{R}$ or $\mathbb{C}$), then the set of $\mathbb{K}$-points $G(\mathbb{K})$, equipped with the analytic manifold topology, is a Lie group. Conversely, an arbitrary Lie group need not admit the structure of an affine algebraic group. In this paper, Lie groups are always considered to arise from such affine algebraic group structure.}, \textit{i.e.}~its coordinate ring $\mathcal{O}(G) \cong \mathbb{K}[G]$. \cite{Drinfeld:1986in, Chari1994, Majid1995}. This algebra of regular functions $\mathcal{O}(G)$ canonically embeds as a subalgebra of the algebra of smooth functions $C^\infty(G)$ in the usual way when $\mathbb{K} = \mathbb{R}$. In this context, the Hopf algebra dual form was introduced for the first time in \cite{Fronsdal:1991gf, Fronsdal1994}. It can be regarded as the canonical object that condenses the so-called quantum duality principle \cite{Semenov-Tian-Shansky:1992xxt} and, since then, the universal $T$-matrix has been constructed for a number of quantum algebras and groups, thereby establishing a comprehensive and explicit framework for Hopf algebra duality, with applications in the construction of $q$-coherent states, the derivation of quantum universal $R$-matrices, as well as in quantum gravity and representation theory (see \cite{Bonechi:1993sn, Jurco:1994cx, Damaskinsky:1994au, Morozov:1994ab, Chakrabarti:1994hp, Jagannathan:1994cm, Ballesteros1995, Fronsdal1995a, Fronsdal1995b, VanDerJeugt:1995yn, Chakrabarti1996, Ballesteros:1996awf, Feinsilver1997, Quesne1998, Ahmedov1999, Chakrabarti1999, Aizawa2000, Sokolov2000, Aizawa:2005fq, Aizawa2006, Aizawa2009, Girelli:2022foc, Mertens:2022aou} and references therein).

The Hopf algebra dual form for quantum deformations of Lie algebras and Lie groups becomes particularly interesting since the universal $T$-matrix can be considered as the noncommutative algebraic analogue of the exponential map for Lie groups. In fact, if $H = U(\mathfrak{g})$, then $H^* = \mathcal{O}(G)$, and Eq.~\eqref{T0} results in an algebraic realization of the usual exponential map from the Lie algebra $\mathfrak{g}$ to the Lie group $G$. That is, if the Lie algebra $\mathfrak{g}$ has basis $\{X^1, \dots, X^n\}$, the universal $T$-matrix reads
\begin{equation} \label{Texp}
	T =  e^{x_1 \otimes X^1} \, e^{x_2\otimes X^2} \, \dots \, e^{x_n \otimes X^n} \ ,
\end{equation}
where the (commutative) local coordinate functions $\{x_1, \dots, x_n\}$ on the Lie group are the dual ones to the corresponding Lie algebra generators. Now, if $H = U_q(\mathfrak{g})$ is a Hopf algebra deformation of $U(\mathfrak{g})$ generated by $\{X^1, \dots, X^n\}$, \textit{i.e.}~a quantum algebra, then $H^* = \mathcal{O}_q(G)$ is the dual Hopf algebra deformation of $\mathcal{O}(G)$ generated by the (noncommutative) local coordinate functions $\{x_1, \dots, x_n\}$, \textit{i.e.}~the quantum group, and the expression for the universal $T$-matrix generally reads
\begin{equation} \label{Tqexp}
	T_q = e_{\alpha_1}^{x_1 \otimes X^1} \, e_{\alpha_2}^{x_2\otimes X^2} \, \dots \, e_{\alpha_n}^{x_n \otimes X^n} \ ,
\end{equation}
where each $e_{\alpha_i}$ can be given, in principle, by a different type of $q$-exponential function, with $\alpha_i$ a function of $q$. We anticipate that the cases considered in this paper will be such that $\alpha_i(q) =  1$, then the universal $T$-matrix will be formally identical to Eq.~\eqref{Texp}, namely written in terms of products of usual exponential functions with exponents in $H^* \otimes H$.

As a consequence, the Hopf algebra dual form is a convenient mathematical object permitting the presentation of quantum groups in the closest possible way to that of Lie groups in a neighborhood of the identity, where the description of the former in terms of `quantum local coordinates' is manifest. Indeed, Eq.~\eqref{Texp} is recovered when taking the limit $q \to 0$ of Eq.~\eqref{Tqexp}, namely in the limit of no quantum deformation. As we will see, the universal $T$-matrix is suitable for analyzing quantum kinematical groups (specifically quantum Galilei and Poincaré groups), since their noncommutative homogeneous spacetimes and their quantum symmetry transformations can be presented in a very transparent way with respect to their `classical' counterparts.

This paper is motivated by a novel application of Hopf algebra dual forms of quantum kinematical groups, which has emerged recently within the framework of `quantum reference frames' (QRFs) \cite{Ballesteros:2025ypr}. The claim that physical reference frames constitute dynamical material systems themselves and therefore require a description within quantum theory has a long history \cite{DeWitt1967, Aharonov1967a}. It has been developed from different perspectives including quantum foundations, quantum information, quantum gravity, and mathematical physics (see, for instance, \cite{Aharonov1984, Rovelli:1990pi, Brown:1995fj, Poulin:2006ryq, Bartlett2007, Girelli:2007xn, Angelo2011, Tambornino:2011vg, Palmer:2013zza, Busch2016, Smith2016, Giacomini:2017zju, Giacomini:2018gxh, Lizzi:2018qaf, Lizzi:2019wto, delaHamette:2020dyi, Hoehn:2020epv, Streiter:2021kta, Krumm:2020fws, Ballesteros:2020lgl, Giacomini:2021gei, Mikusch:2021kro, Lizzi:2022hcq, Apadula:2022pxk, Glowacki:2023nnf, Lake:2023nua, Fewster:2024pur, AliAhmad:2024vdw, DEsposito:2024wru, Fiore:2025oks, DeVuyst:2025ezt, Fedida:2025viy}).

In the perspectival approach to QRFs introduced in \cite{Giacomini:2017zju}, transformations between inertial QRFs describing a quantum particle $B$ are formally realized as usual Galilei transformations where the group parameters are promoted to quantum operators acting on the supplementary Hilbert space of another quantum particle $A$ playing the role of QRF. This idea of `noncommutative group parameters' within Galilei transformations is indeed behind the Hopf algebra dual form \eqref{Texp} for quantum Galilei groups and, in \cite{Ballesteros:2025ypr}, the quantum Galilei group approach to the Galilei QRF transformations of \cite{Giacomini:2017zju} has been developed and related to the Lie group-theoretical perspective of \cite{Ballesteros:2020lgl}. In particular, there exists a unique quantum (1+1) centrally extended Galilei group (with non-trivial central extension) identified in \cite{Ballesteros:2025ypr} such that its universal $T$-matrix produces the sought Galilei QRF transformations, provided that only the first order in the quantum deformation parameter is regarded and that the QRF is in a quantum superposition of semiclassical states. This result triggered the conjecture that the all-order Galilei $T$-matrix should supply the algebraic description of Galilei QRF transformations involving arbitrary quantum states for the QRF (see \cite{Ballesteros:2025ypr} for details).

The aim of this paper is to pave the way for the algebraic generalization of the previous result to the relativistic scenario by finding the quantum deformation of the (1+1) centrally extended Poincaré Lie algebra (with trivial central extension) whose non-relativistic limit gives the quantum Galilei algebra specified in \cite{Ballesteros:2025ypr}. Afterwards, the associated Poincaré $T$-matrix will be computed, and the contraction theory of Hopf algebra dual forms will be developed to obtain the Galilei $T$-matrix of \cite{Ballesteros:2025ypr} as its non-relativistic limit. It is worth stressing that, contrary to the Galilei case \cite{Opanowicz:1998zz, Opanowicz:1999qp, Ballesteros:1999ew}, no complete classification of the quantum deformations of the (1+1) centrally extended Poincaré Lie group exists in the literature (see \cite{Ballesteros:2018ghw} for partial results), hence the quest for the suitable one must be carried out from first principles. It should also be noted that the contraction theory of Lie bialgebras, Poisson--Lie groups, quantum algebras, and quantum groups is well-established \cite{Celeghini:1990bf, Celeghini:1990xx, Celeghini1992, Ballesteros:1994iv, Ballesteros:1994at}, however, to the best of our knowledge, its generalization to Hopf algebra dual forms is presented here for the first time.

The paper is structured as follows: In Section \ref{Sec2}, the basics of Hopf algebra dual forms are given. The case of $U(\mathfrak{g})$ is presented as the prominent example, and the universal $T$-matrix for the (1+1) timelike $\kappa$-Poincaré quantum algebra (without central extension) is computed in detail for the first time (the (1+1) and (3+1) lightlike cases were considered in \cite{Ballesteros1995, Ballesteros:1996awf}, respectively). The contraction theory of Hopf algebra dual forms is introduced in Section \ref{Sec3}. In Section \ref{Sec4}, the quantum deformation of the (1+1) centrally extended Poincaré Lie algebra is derived. Firstly, its Lie bialgebra structure is identified using the Lie bialgebra associated with the quantum Galilei algebra of \cite{Ballesteros:2025ypr} as a guideline. Secondly, the dual Poisson--Lie group quantization approach considered in \cite{Ballesteros2013} is employed to obtain the quantum Poincaré algebra. Finally, the Poincaré $T$-matrix is constructed, and the quantum Poincaré group is obtained. This computation is deferred to the \nameref{Appendix} to improve readability. In the adequate basis, this dual Hopf algebra is identified, interestingly, as a non-trivial central extension of the (1+1) spacelike $\kappa$-Poincaré quantum group. The contraction theory of Hopf algebra dual forms is then applied in Section \ref{Sec5} in order to prove that the Galilei $T$-matrix of \cite{Ballesteros:2025ypr} follows from the non-relativistic limit of the Poincaré $T$-matrix. Section \ref{Sec6} closes the paper with several comments and directions for future research.


\section{The universal \texorpdfstring{$T$}{T}-matrix} \label{Sec2}

The construction developed in the following is valid for any canonically paired dual Hopf algebras but the case where the Hopf algebra $H$ is the universal enveloping algebra $U(\mathfrak{g})$ of the Lie algebra $\mathfrak{g}$ of a Lie group $G$, then the dual Hopf algebra $H^*$ is the algebra of regular functions $\mathcal{O}(G)$ on the Lie group $G$, is particularly relevant, as previously stressed. Let us recall that, in the case of a Lie algebra, the universal $T$-matrix when considering its universal enveloping algebra as a cocommutative Hopf algebra results in the algebraic analogue of the usual exponential map from the Lie algebra to the Lie group and, in the case of a quantum algebra, in the sense of a Hopf algebra deformation of the universal enveloping algebra of a Lie algebra, the universal $T$-matrix can be interestingly understood as the quantum version of this exponential map.


\subsection{Universal \texorpdfstring{$T$}{T}-matrix for \texorpdfstring{$U(\mathfrak{g})$}{U(g)}}

Let $H$ be an arbitrary Hopf algebra over a field $\mathbb{K}$ with basis elements $X^\mu \in H$, and $H^*$ the Hopf algebra dual with basis elements $x_\mu \in H^*$. The dual pairing $\langle \cdot \,, \cdot \rangle: H^* \times H \rightarrow \mathbb{K}$ is such that
\begin{equation} \label{Duality}
     \langle x_\mu, X^\nu \rangle = \delta_\mu^\nu \ .
\end{equation}
As dual Hopf algebras, the product and coproduct in $H$ read
\begin{align}
    X^\mu X^\nu &= E_\rho^{\mu \nu} X^\rho \ , \label{EX} \\
    \Delta(X^\rho) &= F_{\mu \nu}^\rho X^\mu \otimes X^\nu \ , \label{FX}
\end{align}
while in $H^*$ read
\begin{align}
    x_\mu x_\nu &= F_{\mu \nu}^\rho x_\rho \ , \label{Fx} \\
    \Delta(x_\rho) &= E_{\rho}^{\mu \nu} x_\mu \otimes x_\nu \ . \label{Ex}
\end{align}
Notice that the structure tensor $E$ encodes the duality between the product in $H$ and the coproduct in $H^*$, namely
\begin{equation}
    \langle x_\rho, X^\mu X^\nu \rangle = E_\rho^{\mu \nu} = \langle \Delta(x_\rho), X^\mu \otimes X^\nu \rangle \ ,
\end{equation}
and that the structure tensor $F$ encodes the duality between the product in $H^*$ and the coproduct in $H$, namely
\begin{equation}
    \langle x_\mu x_\nu, X^\rho \rangle = F_{\mu \nu}^\rho = \langle x_\mu \otimes x_\nu, \Delta(X^\rho) \rangle \ .
\end{equation}
As a consequence of this duality of Hopf algebras, the universal $T$-matrix $T \in H \otimes H^*$ possesses the following property:
\begin{equation} \label{T-duality}
    T_{x', X} \cdot T_{x'', X} = T_{\Delta(x), X} \ , \quad T_{x, X'} \cdot T_{x, X''} = T_{x, \Delta(X)} \ ,
\end{equation}
where $T_{x, X} = T$ as defined in Eq.~\eqref{T0}, with $x' = x \otimes 1$, $x'' = 1 \otimes x$, $X' = X \otimes 1$, and $X'' = 1 \otimes X$.

Analogous relations concerning the remaining Hopf algebra maps are derived similarly. All Hopf algebras considered in this paper are specified in terms of the coproduct map and the commutation relations encoding the product map, with the remaining structures uniquely determined by the Hopf algebra axioms.

Let us now consider a $n$-dimensional real Lie algebra $\mathfrak{g} = \text{Lie}(G)$ with basis $\{X^1, X^2, \dots, X^n\} \allowbreak \subset \mathfrak{g}$. Its universal enveloping algebra $U(\mathfrak{g})$ can be canonically endowed with a cocommutative Hopf algebra structure. The basis elements of $U(\mathfrak{g})$ are, by the Poincaré--Birkhoff--Witt theorem \cite{Poincare1900, Birkhoff1937, Witt1937}, of the form $X^\mu := X^{\mu_1 \mu_2 \dots \mu_n} = (X^1)^{\mu_1} \, (X^2)^{\mu_2} \, \dots \, (X^n)^{\mu_n} \in U(\mathfrak{g)}$, with the exponents $\mu_1, \mu_2, \dots , \mu_n$ nonnegative integers. The Hopf algebra dual corresponds to the algebra of regular functions $\mathcal{O}(G)$ on the finite-dimensional real Lie group $G$. The basis elements of $\mathcal{O}(G)$ can be expressed as $x_\mu := x_{\mu_1 \mu_2 \dots \mu_n} \in \mathcal{O}(G)$. They can be defined by Hopf algebra duality, namely by the canonical pairing $\langle \cdot \,, \cdot \rangle: \mathcal{O}(G) \times U(\mathfrak{g}) \rightarrow \mathbb{R}$ such that
\begin{equation}
     \langle x_\mu, X^\nu \rangle = \delta_\mu^\nu := \delta_{\mu_1}^{\nu_1} \, \delta_{\mu_2}^{\nu_2} \, \dots \delta_{\mu_n}^{\nu_n} \ .
\end{equation}

In this case, this canonical pairing $\langle \cdot \,, \cdot \rangle: \mathcal{O}(G) \times U(\mathfrak{g}) \rightarrow \mathbb{R}$ is explicitly given by
\begin{equation}
    \langle f \,, X \rangle = (\alpha(X) f)_e \ ,
\end{equation}
for all $X \in U(\mathfrak{g})$ and $f \in \mathcal{O}(G)$, where $e$ is the identity element of the Lie group $G$ and $\alpha$ is the unique anti-homomorphism $\alpha: U(\mathfrak{g}) \rightarrow \text{End}(\mathcal{O}(G))$, which is an extension of the anti-homomorphism $\alpha: \mathfrak{g} \rightarrow \text{End}(\mathcal{O}(G))$ that assigns to an element $X \in \mathfrak{g} \cong T_e G$ its associated right-invariant vector field, namely
\begin{equation}
    (\alpha(X))_g = X_g^R := ((R_g)_*)_e X \ ,
\end{equation}
where $R_g: G \rightarrow G$ is the right-multiplication by $g \in G$ and $((R_g)_*)_e: T_e G \rightarrow T_g G$ is the pushforward of $R_g$ at the identity. As a consequence, through $\alpha$, $X \in \mathfrak{g}$ acts as a differential operator on $\mathcal{O}(G)$ as
\begin{equation}
    X_g^R f = \left. \frac{d}{d \varepsilon} f(e^{\varepsilon X} g) \right|_{\varepsilon = 0} \ ,
\end{equation}
and, by extension, $X \in U(\mathfrak{g})$ acts generally as a higher-order differential operator on $\mathcal{O}(G)$.

Let the Lie group $G$ be connected and simply connected for simplicity. An arbitrary element $g \in G$ can be expressed as
\begin{equation}
    e^{\tilde{x}_i X^i} = e^{\tilde{x}_1 X^1 + \tilde{x}_2 X^2 + \dots + \tilde{x}_n X^n} \ ,
\end{equation}
where $\tilde{x}_i \in \mathbb{R}$ are the coordinates on the Lie group manifold and $X^i \in \mathfrak{g}$ are the generators of the Lie algebra $\mathfrak{g} = \text{Lie}(G)$. We can also see the arbitrary element $g \in G$ as an element of $\mathcal{O}(G) \otimes U(\mathfrak{g})$ instead, where $\tilde{x}_i \in \mathcal{O}(G)$ are the coordinate functions on the Lie group such that $\tilde{x}_i(g) \in \mathbb{R}$ and $X^i \in U(\mathfrak{g})$ are the Lie algebra generators embedded in the universal enveloping algebra. Then,
\begin{equation}
    \begin{aligned}
        \langle e^{\tilde{x}_i \otimes X^i} \,, X^\mu \rangle &= (\alpha(X^\mu) e^{\tilde{x}_i \otimes X^i})_e \\
        &= (X^\mu e^{\tilde{x}_i \otimes X^i})_e \\
        &= X^\mu e^{\tilde{x}_i(e) \otimes X^i} \\
        &= X^\mu \ .
    \end{aligned}
\end{equation}
On the other hand,
\begin{equation}
    \begin{aligned}
        \langle T \,, X^\mu \rangle &= \langle x_\nu \otimes X^\nu \,, X^\mu \rangle \\
        &= \langle x_\nu \,, X^\mu \rangle X^\nu \\
        &= \delta_\nu^\mu X^\nu \\
        &= X^\mu \ ,
    \end{aligned}
\end{equation}
hence the universal $T$-matrix for $U(\mathfrak{g})$ is \cite{Fronsdal:1991gf, Fronsdal1994}
\begin{equation} \label{T}
    T = e^{\tilde{x}_i \otimes X^i} \ .
\end{equation}
The Baker--Campbell--Hausdorff formula allows in this case to factorize Eq.~\eqref{T} formally as
\begin{equation}
    T = e^{x_1 \otimes X_1} \, e^{x_2 \otimes X_2} \, \dots \, e^{x_n \otimes X_n} \ ,
\end{equation}
where the $x_i$ are functions of the $\tilde{x}_i$. Finally, since the basis elements of $U(\mathfrak{g})$ are, by the Poincaré--Birkhoff--Witt theorem \cite{Poincare1900, Birkhoff1937, Witt1937}, of the form $X^\mu := X^{\mu_1 \mu_2 \dots \mu_n} = (X^1)^{\mu_1} \, (X^2)^{\mu_2} \, \dots \, (X^n)^{\mu_n}$, then the basis elements of $\mathcal{O}(G)$ are explicitly of the form
\begin{equation}
    x_\mu := x_{\mu_1 \mu_2 \dots \mu_n} = \frac{x_1^{\mu_1}}{\mu_1!} \, \frac{x_2^{\mu_2}}{\mu_2!} \, \dots \, \frac{x_n^{\mu_n}}{\mu_n!} \ .
\end{equation}


\subsection{Universal \texorpdfstring{$T$}{T}-matrix for the (1+1) timelike \texorpdfstring{$\kappa$}{κ}-Poincaré quantum algebra}

As already commented in the Introduction, the Hopf algebra dual form can consequently be considered the canonical presentation of a quantum group. The universal $T$-matrix exhibits indeed a ``group-like'' property in a representation-independent way in the sense of Eq.~\eqref{T-duality}.

We compute in detail in the following the universal $T$-matrix for a particular Hopf algebra deformation of physical relevance: the (1+1)-dimensional timelike $\kappa$-Poincaré quantum algebra \cite{Lukierski:1991pn, Lukierski:1991ff, Giller:1992xg, Lukierski:1992dt, Zakrzewski:1994hlc}. This case will explicitly illustrate how to construct Hopf algebra dual forms in the quantum group setting. We begin by considering the (1+1) Poincaré Lie algebra,
\begin{equation} \label{Poincaré LA}
    [P_0, P_1] = 0 \ , \quad [K, P_0] = P_1 \ , \quad [K, P_1] = P_0 \ ,
\end{equation}
with $P_0$, $P_1$, and $K$ representing the generators of time translations, space translations, and boosts, respectively. The case we will explicitly work out is the Hopf algebra deformation of the universal enveloping algebra of the (1+1) Poincaré Lie algebra obtained as the quantization of the coboundary Lie bialgebra generated by the standard $r$-matrix $r = \omega K \wedge P_1$. The commutation relations are
\begin{equation}
    [P_0, P_1] = 0 \ , \quad [K, P_0] = P_1 \ , \quad [K, P_1] = \frac{\sinh{(2 \omega P_0)}}{2 \omega} \ ,
\end{equation}
and the compatible coproducts are
\begin{equation}
    \begin{aligned}
        \Delta(P_0) &= P_0 \otimes 1 + 1 \otimes P_0 \ , \\
        \Delta(P_1) &= P_1 \otimes e^{\omega P_0} + e^{- \omega P_0} \otimes P_1 \ , \\
        \Delta(K) &= K \otimes e^{\omega P_0} + e^{- \omega P_0} \otimes K \ .
    \end{aligned}
\end{equation}
Notice that, in this basis, the exponentials appear on both sides of the tensor products in the deformed coproducts. It is convenient, however, to introduce the following change of basis:
\begin{equation}
    P_0 \rightarrow P_0 \ , \quad P_1 \rightarrow e^{\omega P_0} P_1 \ , \quad K \rightarrow e^{\omega P_0} K \ ,
\end{equation}
so that, in this basis, the commutation relations read
\begin{equation} \label{Poincaré QA(a)}
    [P_0, P_1] = 0 \ , \quad [K, P_0] = P_1 \ , \quad [K, P_1] = \frac{e^{4 \omega P_0} - 1}{4 \omega} + \omega P_1^2 \ ,
\end{equation}
and the coproducts read
\begin{equation} \label{Poincaré QA(b)}
    \begin{aligned}
        \Delta(P_0) &= P_0 \otimes 1 + 1 \otimes P_0 \ , \\
        \Delta(P_1) &= P_1 \otimes e^{2 \omega P_0} + 1 \otimes P_1 \ , \\
        \Delta(K) &= K \otimes e^{2 \omega P_0} + 1 \otimes K \ ,
    \end{aligned}
\end{equation}
where the exponentials appear or on the left-hand side or the right-hand side of the tensor products in the deformed coproducts.

The basis elements of this quantum universal enveloping algebra are chosen in the following order:
\begin{equation}
    X^{a b c} = P_0^a P_1^b K^c \ ,
\end{equation}
with the dual basis elements defined by Eq.~\eqref{Duality}. In particular, $X^{1 0 0} = P_0$, $X^{0 1 0} = P_1$, $X^{0 0 1} = K$, so that the dual coordinate functions on the quantum group are defined locally as $x_{1 0 0} = a_0$, $x_{0 1 0} = a_1$, $x_{0 0 1} = \chi$. We have that
\begin{equation}
    \begin{aligned}
        \Delta(X^{a b c}) &= \Delta(P_0) \Delta(X^{(a - 1) b c}) \ , \\
        \Delta(X^{a b c}) &= \Delta(P_1) \Delta(X^{a (b - 1) c}) \ ,  \\
        \Delta(X^{a b c}) &= \Delta(X^{a b (c - 1)}) \Delta(K) \ ,
    \end{aligned}
\end{equation}
which yield the following recurrence relations for the structure tensor $F$ from Eq.~\eqref{FX}:
\begin{equation} \label{F1}
    \begin{aligned}
        F_{i j k ; p q r}^{a b c} &= F_{(i - 1) j k ; p q r}^{(a - 1) b c} + F_{i j k ; (p - 1) q r}^{(a - 1) b c} \ , \\
        F_{i j k ; p q r}^{a b c} &= \sum_{n = 0}^p F_{i (j - 1) k ; n q r}^{a (b - 1) c} \frac{(2 \omega)^{p - n}}{(p - n)!} + F_{i j k ; p (q - 1) r}^{a (b - 1) c} \ . \\
    \end{aligned}
\end{equation}
The recurrence relation coming from $\Delta(X^{a b c}) = \Delta(X^{a b (c - 1)}) \Delta(K)$ is very difficult to find in general. We can show, specifically, that 
\begin{equation} \label{F2}
    \begin{aligned}
        F_{1 0 0 ; p q r}^{a b c} &= a \, \delta_{p + 1}^a \, \delta_q^b \, \delta_r^c \ , \\
        F_{i j k ; 0 1 0}^{a b c} &= b \, \delta_i^a \, \delta_{j + 1}^b \, \delta_k^c \ , \\
        F_{i j k ; 0 0 1}^{a b c} &= c \, \delta_i^a \, \delta_j^b \, \delta_{k + 1}^c \ . \\
    \end{aligned}
\end{equation}

The above components of the structure tensor $F$ allow us to compute the dual basis. Note that, taking into account Eq.~\eqref{Fx} together with Eq.~\eqref{F2},
\begin{equation}
    x_{1 0 0} x_{(p - 1) q r} = F_{1 0 0 ; (p - 1) q r}^{a b c} x_{a b c} = p \, x_{p q r} \ ,
\end{equation}
hence
\begin{equation}
    x_{p q r} = \frac{a_0}{p} x_{(p - 1) q r} = \dots = \frac{a_0^p}{p!} x_{0 q r} \ .
\end{equation}
We obtain, repeating the same strategy, that the dual basis elements are given by
\begin{equation}
	x_{a b c} = \frac{a_0^a}{a!} \frac{a_1^b}{b!} \frac{\chi^c}{c!} \ .
\end{equation}
All of this allows to demonstrate the following proposition:
\begin{proposition}
    The universal $T$-matrix \eqref{T0} for the (1+1) timelike $\kappa$-Poincaré quantum algebra is
    \begin{equation} \label{Poincaré T}
        T = e^{a_0 \otimes P_0} e^{a_1 \otimes P_1} e^{\chi \otimes K} \ .
    \end{equation}
\end{proposition}

As already commented, by considering this example, we explicitly see that the construction of the Hopf algebra dual form is just the quantum group analogue of the exponential map from the Lie algebra to the Lie group, and the universal $T$-matrix \eqref{Poincaré T} is simply the timelike $\kappa$-Poincaré quantum group version of the Lie group element in a neighborhood of the identity.

We can again consider Eq.~\eqref{Fx} to derive the commutation relations for the coordinates on the dual quantum group as
\begin{equation}
    [x_{i j k}, x_{p q r}] = (F_{i j k ; p q r}^{a b c} - F_{p q r ; i j k}^{a b c}) x_{a b c} \ .
\end{equation}
The commutation relation $[a_0, a_1]$ is obtained from the recurrence relations \eqref{F1}, but the remaining two, $[\chi, a_0]$ and $[\chi, a_1]$, require the third recurrence relation, which is hard to find in general. There is no other way but looking for the specific components of the structure tensor $F$ involved in such commutation rules.

In the case of $[\chi, a_0]$, $F_{0 0 1 ; 1 0 0}^{a b c}$ is, by duality, the coefficient of the term $K \otimes P_0$ in the coproduct of $X^{a b c}$. This term appears when $a = 1$, $b = 0$, $c = 1$, as well as when $a = 0$, $b = 0$, and $c$ arbitrary. One can show that $F_{0 0 1 ; 1 0 0}^{1 0 1} = 1$ and $F_{0 0 1 ; 1 0 0}^{0 0 m} = 2 \omega$ for $m = 2 n + 1$ (with $n = 0, 1, 2, \dots$). On the other hand, $F_{1 0 0 ; 0 0 1}^{a b c}$ is given by Eq.~\eqref{F2}.

In the case of $[\chi, a_1]$, $F_{0 0 1 ; 0 1 0}^{a b c}$ is, by duality, the coefficient of the term $K \otimes P_1$ in the coproduct of $X^{a b c}$. This term appears when $a = 0$, $b = 1$, $c = 1$, as well as when $a = 0$, $b = 0$ and $c$ arbitrary. One can show that $F_{0 0 1 ; 0 1 0}^{0 1 1} = 1$ and $F_{0 0 1 ; 0 1 0}^{0 0 m} = 2 \omega$ for $m = 2 n$ (with $n = 1, 2, \dots$). On the other hand, $F_{0 1 0 ; 0 0 1}^{a b c}$ is given by Eq.~\eqref{F2}.

We finally end up with the following commutation relations for the coordinate functions on the quantum group:
\begin{equation} \label{Poincaré QG(a)}
    [a_0, a_1] = - 2 \omega a_1 \ , \quad [\chi, a_0] = 2 \omega \sinh\chi \ , \quad [\chi, a_1] = 2 \omega (\cosh\chi - 1) \ .
\end{equation}
The coproducts could also be derived similarly. We can, however, take advantage of the universal $T$-matrix \eqref{Poincaré T}. In particular, by taking into account the characteristic property of Hopf algebra dual forms described in Section \ref{Sec2}, the coproducts will follow from the corresponding ``group-like'' property of the universal $T$-matrix in the sense of Eq.~\eqref{T-duality}. Its application is simpler when realizing the universal $T$-matrix as a true matrix since it reduces to a matrix multiplication. We will consider this approach and introduce the fundamental representation of the (1+1) timelike $\kappa$-Poincaré algebra \eqref{Poincaré QA(a)} given by\footnote{This is also the fundamental representation of the (1+1) Poincaré Lie algebra \eqref{Poincaré LA}.}
\begin{equation}
    \rho(P_0) = \left(\begin{array}{ccc}
    0 & 0 & 0 \\
    1 & 0 & 0 \\
    0 & 0 & 0
    \end{array}\right) \ , \quad 
    \rho(P_1) = \left(\begin{array}{ccc}
    0 & 0 & 0 \\
    0 & 0 & 0 \\
    1 & 0 & 0
    \end{array}\right) \ , \quad 
    \rho(K) = \left(\begin{array}{ccc}
    0 & 0 & 0 \\
    0 & 0 & 1 \\ 
    0 & 1 & 0
    \end{array}\right) \ .
\end{equation}
As a consequence, the realization of the $T$-matrix \eqref{Poincaré T} is
\begin{equation}
    T^\rho := (1 \otimes \rho) \, T = e^{a_0 \otimes \rho(P_0)} e^{a_1 \otimes \rho(P_1)} e^{\chi \otimes \rho(K)} = \left(\begin{array}{ccc}
    1 & 0 & 0 \\
    a_0 & \cosh\chi & \sinh\chi \\
    a_1 & \sinh\chi & \cosh\chi
    \end{array}\right) \ ,
\end{equation}
which is just the Poincaré Lie group element in a neighborhood of the identity where now entries are noncommutative. It can be considered locally as the matrix representation of the (1+1) timelike $\kappa$-Poincaré quantum group element. The coproduct map for its coordinate functions is read off from the matrix multiplication of two copies of $T^\rho$. The result is
\begin{equation} \label{Poincaré QG(b)}
    \begin{aligned}
        \Delta(a_0) & = a_0 \otimes 1 + \cosh\chi \otimes a_0 + \sinh\chi \otimes a_1 \ , \\
        \Delta(a_1) &= a_1 \otimes1 + \cosh\chi \otimes a_1 + \sinh\chi \otimes a_0 \ , \\
        \Delta(\chi) &= \chi \otimes 1 + 1 \otimes \chi \ ,
    \end{aligned} 
\end{equation}
thus encoding the (classical) Poincaré group composition law for the local coordinates.

The Hopf algebra dual form permits dualizing this procedure. We can instead introduce the fundamental representation of the algebra of (1+1) timelike $\kappa$-Poincaré group coordinate functions \eqref{Poincaré QG(a)} given by
\begin{equation}
    \sigma(a_0) = \left(\begin{array}{ccc}
    0 & 0 & 0 \\
    0 & 2 \omega & 0 \\
    0 & 0 & 2 \omega
    \end{array}\right) \ , \quad
   \sigma(a_1) = \left(\begin{array}{ccc}
    0 & 1 & 0 \\
    0 & 0 & 0 \\
    0 & 0 & 0
    \end{array}\right) \ , \quad
    \sigma(\chi) = \left(\begin{array}{ccc}
    0 & 0 & 1 \\
    0 & 0 & 0 \\
    0 & 0 & 0
    \end{array}\right) \ . 			
\end{equation}
The corresponding realization of the $T$-matrix \eqref{Poincaré T} is
\begin{equation}
    T^{\sigma} := (\sigma \otimes 1) \, T = e^{\sigma(a_0) \otimes P_0} e^{\sigma(a_1) \otimes P_1} e^{\sigma(\chi) \otimes K} = \left(\begin{array}{ccc}
    1 & P_1 & K \\
    0 & e^{2 \omega P_0} & 0 \\
    0 & 0 & e^{2 \omega P_0}
    \end{array}\right) \ ,
\end{equation}
which is regarded locally as the matrix representation of the (1+1) timelike $\kappa$-Poincaré dual quantum group element (with noncommutative entries). One can check (by the matrix multiplication of two copies of $T^{\sigma}$) that the coproduct of its coordinate functions, now the quantum algebra generators, is that given in Eq.~\eqref{Poincaré QA(b)}.

The concept of Lie bialgebra will be particularly relevant for the scopes of this work. Although the formal definition is given in the next section, let us now write the Lie bialgebra associated with the $\kappa$-Poincaré quantum algebra considered in this section. This Poincaré Lie bialgebra is characterized by the Lie algebra \eqref{Poincaré LA}, with the Lie bialgebra structure given by a cocommutator map $\delta$. The cocommutator corresponds to (the skew-symmetric part of) the first order of the coproduct in the deformation parameter, namely
\begin{equation}
    \delta = \frac{1}{2} (\Delta - \tau \circ \Delta) \quad \mathrm{mod} \, \alpha^2 \ ,
\end{equation}
where $\tau$ is the flip operator, \textit{i.e.}~$\tau(X \otimes Y) = Y \otimes X$, and with $\alpha$ the deformation parameter. The (1+1) timelike $\kappa$-Poincaré Lie bialgebra has as cocommutator the first order of Eq.~\eqref{Poincaré QA(b)} in $\omega$, namely
\begin{equation} \label{Poincaré LC}
    \begin{aligned}
        \delta(P_0) &= 0 \ , \\
        \delta(P_1) &= \omega P_1 \wedge P_0 \ , \\
        \delta(K) &= \omega K \wedge P_0 \ .
    \end{aligned}
\end{equation}
In the case of coboundary Lie bialgebras, as is this case, the cocommutator map is such that $\delta(X) = \mathrm{ad}_X r, \, \forall X \in \mathfrak{g}$, given an $r$-matrix $r \in \mathfrak{g} \wedge \mathfrak{g}$ solution of the classical Yang--Baxter equation (non-standard) or the modified classical Yang--Baxter equation (standard). Otherwise, Lie bialgebras that do not admit such an $r$-matrix are called non-coboundary.

We recall that the Hopf algebra dual form obtained in this section corresponds to the timelike deformation of the Poincaré Lie algebra generated by the (standard) $r$-matrix $r = \omega K \wedge P_1$. Notice that, in the Lie bialgebra, the cocommutator of the time translation generator is zero, hence the coproduct is primitive in the quantum algebra. The very same construction of this section for the spacelike deformation, the one generated by the (standard) $r$-matrix $r = \omega K \wedge P_0$, is straightforward by swapping the role of $P_0$ and $P_1$. Finally, the last case, the lightlike deformation, coming from the (non-standard) $r$-matrix $r = \omega K \wedge (P_0 + P_1)$, was fully constructed in \cite{Ballesteros1995}.


\section{Inönü--Wigner contractions of the universal \texorpdfstring{$T$}{T}-matrix} \label{Sec3}

Inönü--Wigner contractions \cite{Inonu:1953sp} of quantum algebras and quantum groups are considered in this section. We first recall the theory of Lie bialgebra contractions and Poisson--Lie groups which, based on previous results given in \cite{Celeghini:1990bf, Celeghini:1990xx, Celeghini1992}, was completely developed in \cite{Ballesteros:1994iv, Ballesteros:1994at}. Afterwards, new concepts concerning contractions of multiparametric Lie bialgebras are provided. Finally, we introduce for the first time a formal contraction theory of Hopf algebra dual forms. Throughout this section, we adopt the conventions and notation of \cite{Ballesteros:1994iv, Ballesteros:1994at}.


\subsection{Lie bialgebra contractions}

Let us first recall the concept of contractions of a (finite-dimensional) Lie algebra $\mathfrak{g} := (V, [\cdot, \cdot])$.

\begin{definition}
    A Lie algebra $\mathfrak{g}'$ is a contraction of $\mathfrak{g}$ (with the same underlying vector space $V$) if there exists a one-parametric family of Lie algebra automorphisms $\phi_\varepsilon: \mathfrak{g} \rightarrow \mathfrak{g}$ such that the Lie bracket in $\mathfrak{g}$ defines the Lie bracket in $\mathfrak{g}'$ as the limit
    \begin{equation} \label{Lim1}
        [X, Y]' := \lim_{\varepsilon \to 0} \phi_\varepsilon^{-1} [\phi_\varepsilon(X), \phi_\varepsilon(Y)], \, \forall X, Y \in \mathfrak{g} \ .
    \end{equation}
\end{definition}

We will assume that the automorphisms $\phi_\varepsilon$ are polynomials on the contraction parameter $\varepsilon$, as usual.

We consider the contraction from the (1+1) Poincaré Lie algebra to the (1+1) Galilei Lie algebra as an example in the following. In this case, the contraction parameter $\varepsilon$ will have a particular physical interpretation. It will be related to $c$, \textit{i.e.}~the speed of light. Specifically, $\varepsilon = c^{-1}$, such that the contraction limit $\varepsilon \to 0$ corresponds to the non-relativistic limit $c \to + \infty$.

\begin{example} \label{Ex1}
    The (1+1) Poincaré Lie algebra (Eq.~\eqref{Poincaré LA}) is
    \begin{equation}
        [P_0, P_1] = 0 \ , \quad [K, P_0] = P_1 \ , \quad [K, P_1] = P_0 \ .
    \end{equation}
    The following contraction map is introduced now:
    \begin{equation} \label{C}
        \phi_c(P_0) = P_0 \ , \quad \phi_c(P_1) = c^{-1} P_1 \ , \quad \phi_c(K) = c^{-1} K \ ,
    \end{equation}
    hence Eq.~\eqref{Lim1} reads
    \begin{equation}
        [P_0, P_1]' = \lim_{c \to + \infty} 0 \ , \quad [K, P_0]' = \lim_{c \to + \infty} P_1 \ , \quad [K, P_1]' = \lim_{c \to + \infty} c^{-2} P_0 \ ,
    \end{equation}
    and the limit $c \to + \infty$ renders
    \begin{equation}
        [P_0, P_1]' = 0 \ , \quad [K, P_0]' = P_1 \ , \quad [K, P_1]' = 0 \ ,
    \end{equation}
    which is the (1+1) Galilei Lie algebra.
\end{example}

This contraction procedure for Lie algebras generalizes to Lie bialgebras:

\begin{definition}
    A Lie bialgebra $(\mathfrak{g}, \delta)$ is a Lie algebra $\mathfrak{g}$ with Lie bracket $[\cdot, \cdot]$: $\mathfrak{g} \wedge \mathfrak{g} \rightarrow \mathfrak{g}$ endowed with a cocommutator map $\delta$: $\mathfrak{g} \rightarrow \mathfrak{g} \wedge \mathfrak{g}$ such that:
    \begin{itemize}
        \item[i)] The cocommutator $\delta$ is a 1-cocycle, \textit{i.e.}
        \begin{equation} \label{1-cocycle}
            \delta([X, Y]) = \mathrm{ad}_X \delta(Y) - \mathrm{ad}_Y \delta(X) \ , \quad \forall X, Y \in \mathfrak{g} \ .
        \end{equation}
        \item[ii)] The dual of the cocommutator $\delta^*$: $\mathfrak{g}^* \wedge \mathfrak{g}^* \rightarrow \mathfrak{g}^*$ satisfies the Jacobi identity, so that $\mathfrak{g}^* := (V^*, \delta^*)$ is a Lie algebra. We say that $\mathfrak{g}^*$ is the dual Lie algebra of $\mathfrak{g}$.
    \end{itemize}
\end{definition}

Note that, by construction, the Lie algebra $\mathfrak{g}^*$ is endowed with a cocommutator map $[\cdot, \cdot]^*: \mathfrak{g}^* \rightarrow \mathfrak{g}^* \wedge \mathfrak{g}^*$ whose dual is the Lie bracket $[\cdot, \cdot]: \mathfrak{g} \wedge \mathfrak{g} \rightarrow \mathfrak{g}$, so that $(\mathfrak{g}^*, [\cdot, \cdot]^*)$ is a Lie bialgebra. We say that $(\mathfrak{g}^*, [\cdot, \cdot]^*)$ is the dual Lie bialgebra of $(\mathfrak{g}, \delta)$.

The theory of Lie bialgebra contractions is based on the following proposition:

\begin{proposition} \label{LBC}
    \cite{Ballesteros:1994at} Let $(\mathfrak{g}, \eta)$ be a Lie bialgebra and let the Lie algebra $\mathfrak{g}'$ be a contraction of $\mathfrak{g}$ defined by the map $\phi_\varepsilon$. If $n$ is any real number such that the limit
    \begin{equation} \label{Lim2}
        \eta' := \lim_{\varepsilon \to 0} \varepsilon^n (\phi_\varepsilon^{-1} \otimes \phi_\varepsilon^{-1}) \circ \eta \circ \phi_\varepsilon
    \end{equation}
    exists, then $(\mathfrak{g}', \eta')$ is a Lie bialgebra. Additionally, there exists a minimal fixed value $n_0$ of $n$ such that for $n \geq n_0$ the limit \eqref{Lim2} exists and for $n > n_0$ the limit is zero.
\end{proposition}

Let us now consider the map $\phi_\varepsilon$ defining the Lie algebra contraction from $\mathfrak{g}$ to $\mathfrak{g}'$. Proposition \ref{LBC} provides a family of ``contracted" Lie bialgebras $(\mathfrak{g}', \eta')$ parametrized by the real number $n \geq n_0$. Note that the cocommutator map $\eta'$ is non-trivial only when $n = n_0$. These definitions follow naturally then:

\begin{definition}
    The Lie bialgebra $(\mathfrak{g}', \eta')$ is said to be a Lie bialgebra contraction (LBC) of $(\mathfrak{g}, \eta)$ if there exists a contraction from $\mathfrak{g}$ to $\mathfrak{g}'$ described by a family of automorphisms $\phi_\varepsilon$ and a number $n$ such that $\eta'$ is given by the limit \eqref{Lim2}. This LBC is denoted as $(\phi_\varepsilon, n)$.
\end{definition}

\begin{definition}
    The minimal value $n_0$ of $n$ will be called the fundamental contraction constant of the Lie bialgebra $(\mathfrak{g}, \eta)$ associated with $\phi_\varepsilon$.
\end{definition}

\begin{definition}
    The LBC $(\phi_\varepsilon, n_0)$ will be called the fundamental LBC of the Lie bialgebra $(\mathfrak{g}, \eta)$ associated with $\phi_\varepsilon$.
\end{definition}

The above discussion applies to any Lie bialgebra, either coboundary or non-coboundary. In the case of coboundary Lie bialgebras, the contraction of the classical $r$-matrix generating the cocommutator map $\delta$ may also be considered. This case will not be described in this section, but the procedure is similar, though only in some cases a non-trivial contraction of the classical $r$-matrix is possible. This is natural since not all Lie bialgebras are coboundary, and the contracted Lie bialgebra may be non-coboundary. The reader is referred to \cite{Ballesteros:1994at} for a complete explanation of this problem, which will not be relevant in this work.

Let us stress that the convergence of the limit \eqref{Lim2} is ensured, in general, due to the introduction of $\varepsilon^n$ (see Proposition \ref{LBC}). When a Lie bialgebra is considered as the first-order deformation of a given quantum algebra (see the particular case presented in the previous section), the factor $\varepsilon^n$ can be interpreted as a kind of ``renormalization'' of the deformation parameter coming from the quantum group structure.

Let us consider, more explicitly, a Lie bialgebra $(\mathfrak{g}, \eta)$ arising as the first order of a given quantum algebra with quantum deformation parameter $z$. A one-parametric family of Lie bialgebras $(\mathfrak{g}, \delta)$ can be constructed from it with cocommutator map $\delta := z \eta$. We can then associate the corresponding renormalization factor $\varepsilon^{n_0}$ of the fundamental LBC with a transformation of the deformation parameter under the contraction defined by \cite{Ballesteros:1994at}
\begin{equation} \label{zC}
    w = \phi_\varepsilon(z) := \varepsilon^{- n_0} z \ ,
\end{equation}
where $w$ is the new parameter for the contracted family of Lie bialgebras. As a consequence, the LBC \eqref{Lim2} can be defined without manifestly invoking the factor $\varepsilon^{n_0}$, namely
\begin{equation} \label{Lim3}
    \begin{aligned}
        \delta' &:= \lim_{\varepsilon \to 0} (\phi_\varepsilon^{-1} \otimes \phi_\varepsilon^{-1}) \circ \delta \circ \phi_\varepsilon \\
        &= \lim_{\varepsilon \to 0} \phi_\varepsilon^{-1}(z) (\phi_\varepsilon^{-1} \otimes \phi_\varepsilon^{-1}) \circ \eta \circ \phi_\varepsilon \\
        &= \lim_{\varepsilon \to 0} w \varepsilon^{n_0} (\phi_\varepsilon^{-1} \otimes \phi_\varepsilon^{-1}) \circ \eta \circ \phi_\varepsilon \\
        &= w \eta' \ .
    \end{aligned}
\end{equation}
It is possible to ``lift'' this result to the contraction theory of the quantum algebra from which the Lie bialgebra $(\mathfrak{g}, \eta)$ was derived \cite{Ballesteros:1994at}, since such contraction is well-defined provided that the quantum deformation parameter $z$ is transformed into $w$ everywhere in the Hopf algebra under the contraction.

We consider the contraction from the (1+1) timelike $\kappa$-Poincaré Lie bialgebra to the (1+1) timelike $\kappa$-Galilei Lie bialgebra as an example in the following. The Lie algebra contraction was performed in Example \ref{Ex1}.

\begin{example} \label{Ex2}
    The (1+1) timelike $\kappa$-Poincaré cocommutator (Eq.~\eqref{Poincaré LC}) is
    \begin{equation}
        \begin{aligned}
        \delta(P_0) &= 0 \ , \\
        \delta(P_1) &= \omega P_1 \wedge P_0 \ , \\
        \delta(K) &= \omega K \wedge P_0 \ .
        \end{aligned}
    \end{equation}
    The Lie algebra contraction map  was introduced in Eq.~\eqref{C}, hence Eq.~\eqref{Lim3} reads
    \begin{equation}
        \begin{aligned}
        \delta'(P_0) &= \lim_{c \to + \infty} 0 \ , \\
        \delta'(P_1) &= \lim_{c \to + \infty} \phi_c^{-1}(\omega) P_1 \wedge P_0 \ , \\
        \delta'(K) &= \lim_{c \to + \infty} \phi_c^{-1}(\omega) K \wedge P_0 \ .
        \end{aligned}
    \end{equation}
    so the following deformation parameter transformation determines the fundamental LBC:
    \begin{equation}
        \omega' = \phi_c(\omega) = \omega \ .
    \end{equation}
    Note that, in this case, $n_0 = 0$. The limit $c \to + \infty$ renders
    \begin{equation}
        \begin{aligned}
        \delta'(P_0) &= 0 \ , \\
        \delta'(P_1) &= \omega' P_1 \wedge P_0 \ , \\
        \delta'(K) &= \omega' K \wedge P_0 \ .
        \end{aligned}
    \end{equation}
    which is the (1+1) timelike $\kappa$-Galilei cocommutator. Notice that the $\kappa$-Galilei and $\kappa$-Poincaré cocommutators are identical.
\end{example}


\subsubsection{Multiparametric Lie bialgebra contractions}

Quantum groups with more than one deformation parameter can also be constructed and the corresponding Lie bialgebras are thus multiparametric ones. In fact, multiparametric Lie bialgebras arise naturally in the classification of the quantum deformations of the universal enveloping algebra $U(\mathfrak{g})$ of a given Lie algebra $\mathfrak{g}$ (see \cite{Ballesteros1999, Ballesteros:1999ew} and references therein). This is the approach that we will follow in the next section to obtain a suitable Lie bialgebra structure for the (1+1) centrally extended Poincaré Lie algebra that will eventually give a quantum Poincaré group with prescribed properties under the contraction to a quantum Galilei group. It is then appropriate to discuss contractions of multiparametric Lie bialgebras at this stage. 

Let us consider a multiparametric Lie bialgebra $(\mathfrak{g}, \delta_{z_1, \dots, z_p})$ with $p$ independent deformation parameters $z_i$, in the sense that every limit $z_i \to 0$ of $\delta_{z_1, \dots, z_p}$ produces a well-defined Lie bialgebra structure. 

Let now $\mathfrak{g}'$ be a Lie algebra that can be obtained by contraction of the Lie algebra $\mathfrak{g}$, with contraction map $\phi_\varepsilon$. We can then characterize the fundamental LBC corresponding to each parameter $z_i$, which must be considered once the remaining $z_j$ are set to zero. Note that, in general, each of these `elementary' fundamental LBC $(\phi_\varepsilon, n_{0, i})$, gives rise to a different Lie bialgebra structure on $\mathfrak{g}'$, which depends only on one parameter given by $w_i = \varepsilon^{- n_{0, i}} z_i$. 

We end up with a set of $p$ elementary fundamental LBCs $(\phi_\varepsilon, n_{0, i})$ characterized by the set of $p$ fundamental constants $\{n_{0, 1}, \dots, n_{0, p}\}$. Notice that, if we consider an elementary LBC $(\phi_\varepsilon, n_i)$ such that $n_i > n_{0, i}$, then this elementary LBC will not be fundamental and the contracted Lie bialgebra will be the trivial one with respect to $z_i$. An arbitrary well-defined contraction of the multiparametric Lie bialgebra $(\mathfrak{g}, \delta_{z_1, \dots, z_p})$ is therefore defined by the composition of $p$ elementary LBCs $(\phi_\varepsilon, n_i)$ where the set of transformed parameters is
\begin{equation}
    \{ w_1, \dots, w_p\} = \{\varepsilon^{- n_{1}} z_1, \dots, \varepsilon^{- n_{p}} z_p\} \ ,
\end{equation}
with $n_i \geq n_{0, i}$ for $i = 1, \dots, p$. We can therefore decide, for each of the $p$ elementary LBCs $(\phi_\varepsilon, n_i)$, whether it is fundamental or not. If it is not, the terms in the cocommutator associated with this parameter vanish in the contracted Lie bialgebra. We finally define the fundamental multiparametric LBC:

\begin{definition}
    The fundamental multiparametric LBC is defined by the composition of the $p$ elementary fundamental LBCs $(\phi_\varepsilon, n_{0, i})$.
\end{definition}

As a consequence, the fundamental multiparametric LBC will lead to a unique multiparametric Lie bialgebra $(\mathfrak{g}', \delta_{w_1, \dots, w_p})$ with $p$ deformation parameters $w_i$. This is the most general contracted multiparametric Lie bialgebra obtained under the contraction map $\phi_\varepsilon$.    
    
Alternatively, if a certain number, say $k$, of the constants $n_i$ are selected such that, for them, $n_i > n_{0, i}$, then the multiparametric LBC will be not called fundamental, and the contracted Lie bialgebra will be a multiparametric Lie bialgebra $(\mathfrak{g}', \delta_{w_1, \dots, w_{(p - k)}})$ with $(p - k)$ deformation parameters.

A particular case within the class of non-fundamental multiparametric LBC would be the one such that all the constants $n_i$ are equal, namely $n_i = n$ for $i = 1, \dots, p$. In this case, both the convergence of the multiparametric LBC and the fact that the contracted Lie bialgebra is non-trivial are ensured by choosing the constant $n$ equal to be the maximum $m_0$ of the set of $p$ fundamental constants $\{n_{0, 1}, \dots, n_{0, p}\}$. As a consequence, the cocommutator in the contracted Lie bialgebra will only retain the terms coming from the deformation parameters whose fundamental constant is $n = m_0$ (if $n > m_0$ then the contracted Lie bialgebra will be the trivial one). This non-fundamental multiparametric LBC could be properly referred to as the homogeneous multiparametric LBC.


\subsection{Quantum algebra and quantum group contractions}

We recall that the condition of the deformation parameter transformation previously considered was discovered when quantum deformations of non-semisimple Lie algebras were obtained through a contraction procedure for the first time in \cite{Celeghini:1990bf, Celeghini:1990xx, Celeghini1992}. It was afterwards demonstrated in \cite{Ballesteros:1994at} that finding the fundamental LBC of the associated Lie bialgebra gives a unique prescription for the definition of the transformation of the deformation parameter that turned out to be sufficient to guarantee a non-trivial contraction of the quantum algebra.

Let us consider a quantum algebra $(U_z(\mathfrak{g}), \Delta_z)$ whose associated uniparametric family of Lie bialgebra structures obtained from the first order of the coproduct map in the deformation parameter is $(\mathfrak{g}, \delta_z)$. Let $\phi_\varepsilon$ be the contraction map that defines the Lie algebra contraction from $\mathfrak{g}$ to $\mathfrak{g}'$ with the same dimension, then there exists the fundamental LBC $(\phi_\varepsilon, n_0)$ from $(\mathfrak{g}, \delta_z)$ to $(\mathfrak{g}', \delta_w)$. One can show \cite{Ballesteros:1994at} that the contracted commutation relations and coproduct in the contracted quantum algebra $(U_w(\mathfrak{g}'), \Delta_w)$ are given by
\begin{equation}
    [X, Y]_w := \lim_{\varepsilon \to 0} \phi_\varepsilon^{-1} [\phi_\varepsilon(X), \phi_\varepsilon(Y)]_z \ ,
\end{equation}
and
\begin{equation}
    \Delta_w := \lim_{\varepsilon \to 0} (\phi_\varepsilon^{-1} \otimes \phi_\varepsilon^{-1}) \circ \Delta_z \circ \phi_\varepsilon \ ,
\end{equation}
provided that, under contraction, the deformation parameter transforms as in Eq.~\eqref{zC}. The reader is referred to \cite{Ballesteros:1994at} for more details.

The contraction procedure for quantum groups is based on the fact that quantum groups are Hopf algebra duals of quantum algebras, and by taking into account that a contraction map for the local coordinate functions on a Lie group can be defined in terms of the contraction map for its Lie algebra.

Let us now consider, specifically, a finite-dimensional real Lie algebra $\mathfrak{g} = \text{Lie}(G)$ with basis elements $X^i \in \mathfrak{g}$. The coordinate functions $x_i \in \mathcal{O}(G)$ on the Lie group $G$ are defined by duality, namely
\begin{equation} \label{DualityLG}
    \langle x_i, X^j \rangle = \delta_i^j \ ,
\end{equation}
which is a particular case of the complete Hopf algebra duality \eqref{Duality}. Let $\phi_\varepsilon(X^i)$ be a contraction map defined on the Lie algebra generators $X^i \in \mathfrak{g}$. We can write
\begin{equation}
    \phi_\varepsilon(X^i) = (\phi_\varepsilon)_j^i X^j \ ,
\end{equation}
with $(\phi_\varepsilon)_i^j$ the matrix elements of $\phi_\varepsilon$. Let now $\varphi_\varepsilon(x_i)$ be a contraction map defined on the Lie group coordinates $x_i \in \mathcal{O}(G)$. Similarly,
\begin{equation}
    \varphi_\varepsilon(x_i) = (\varphi_\varepsilon)_i^j x_j \ .
\end{equation}
If we want the contraction maps $\phi_\varepsilon$ and $\varphi_\varepsilon$ to preserve duality, \textit{i.e.}~$\langle \varphi_\varepsilon(x_i),  \phi_\varepsilon(X^j) \rangle = \delta_i^j,$ we end up with the following requisite:
\begin{equation} \label{dLBC}
    (\varphi_\varepsilon)_i^k (\phi_\varepsilon)_k^j = \delta_i^j \ ,
\end{equation}
which implies that the matrix defining the Lie group contraction is the inverse of the one defining the Lie algebra contraction.

As already commented, when considering quantum algebras and quantum groups, the cocommutative Hopf algebra generated by $X^i \in U(\mathfrak{g})$ is deformed into a non-cocommutative Hopf algebra with the same generators $X^i \in U_z(\mathfrak{g})$ and the commutative Hopf algebra generated by $x_i \in \mathcal{O}(G)$ is deformed into a noncommutative Hopf algebra with the same generators $x_i \in \mathcal{O}_z(G)$, respectively, in such a way that the duality \eqref{DualityLG} is preserved as embedded into the full Hopf algebra duality \eqref{Duality} and, with it, the contraction maps $\phi_\varepsilon(X^i)$ and $\varphi_\varepsilon(x_i)$ remain undeformed. On the other hand, the deformation parameter transformation under contraction will be uniquely defined by the fundamental LBC used for the contraction of the quantum algebra.

As a consequence, once the LBC is defined, and with it the contraction of the corresponding quantum algebra, the contraction of the dual Hopf algebra (quantum group) is defined entirely by Eq.~\eqref{dLBC} together with the transformation of the deformation parameter $\phi_\varepsilon(z) = \varphi_\varepsilon(z)$ given by Eq.~\eqref{zC}, which works simultaneously in the quantum algebra and quantum group structures. The above discussion also provide the basics for the contraction theory of Poisson--Lie groups, which are the semiclassical counterparts of quantum groups, whose first examples are given in \cite{Ballesteros:1994iv}.


\subsection{Universal \texorpdfstring{$T$}{T}-matrix contractions}

We can now proceed to introduce the contraction theory of the corresponding Hopf algebra dual forms, which collect the quantum algebra and the quantum group in a single object that algebraically generalizes the exponential map for Lie groups. 

We recall that the universal $T$-matrix was defined in Eq.~\eqref{T0}, where the $X^\mu \in U_z(\mathfrak{g})$ are basis elements of the quantum algebra, \textit{i.e.}~the Hopf algebra deformation of the universal enveloping algebra of the Lie algebra $\mathfrak{g}$ of a Lie group $G$, and the $x_\mu \in \mathcal{O}_z(G)$ are basis elements of the quantum group, \textit{i.e.}~the Hopf algebra deformation of the algebra of regular functions on the Lie group $G$. If the Lie algebra $\mathfrak{g}$ is $n$-dimensional, with basis $\{X^1, X^2, \dots, X^n\} \in \mathfrak{g}$, then $X^\mu := X^{\mu_1 \mu_2 \dots \mu_n} = (X^1)^{\mu_1} \, (X^2)^{\mu_2} \, \dots \, (X^n)^{\mu_n} \in U_z(\mathfrak{g)}$ and $x_\mu := x_{\mu_1 \mu_2 \dots \mu_n} \in \mathcal{O}_z(G)$ (with $\mu_1, \mu_2, \dots , \mu_n$ nonnegative integers) is defined by the duality relation \eqref{Duality}. We have then the following definition for the contraction of the corresponding Hopf algebra dual form:

\begin{definition}
    The universal $T$-matrix $T' \in \mathcal{O}_z(G') \otimes U_z(\mathfrak{g'})$ is a contraction of $T \in \mathcal{O}_z(G) \otimes U_z(\mathfrak{g})$ if there exists the limit
    \begin{equation} \label{Lim4}
        T' := \lim_{\varepsilon \to 0} (\varphi_\varepsilon \otimes \phi_\varepsilon) T = \lim_{\varepsilon \to 0} \varphi_\varepsilon(x_\mu) \otimes \phi_\varepsilon(X^\mu) \ ,
    \end{equation}
where $\phi_\varepsilon$ is the Lie algebra contraction map from $\mathfrak{g}$ to $\mathfrak{g'}$ extended to $U_z(\mathfrak{g})$, $\varphi_\varepsilon$ is the contraction map determined by the preservation of the Hopf algebra duality \eqref{Duality}, and the deformation parameter transformation is given by $\phi_\varepsilon(z) = \varphi_\varepsilon(z)$. 
\end{definition}

Note that the contraction maps $\phi_\varepsilon$ and $\varphi_\varepsilon$ must satisfy $\langle \varphi_\varepsilon(x_\mu), \phi_\varepsilon(X^\nu) \rangle = \delta_{\mu_1}^{\nu_1} \, \delta_{\mu_2}^{\nu_2} \, \dots \delta_{\mu_n}^{\nu_n}$. This is automatically satisfied provided that the relation \eqref{dLBC} is verified.

As is well-known, obtaining the universal $T$-matrix for a given quantum group is non-trivial. In all the cases presented in this work, the dual basis elements $x_\mu \in \mathcal{O}_z(G)$ will be of the form
\begin{equation} \label{dBasis}
    x_\mu = \frac{x_1^{\mu_1}}{\mu_1!} \, \frac{x_2^{\mu_2}}{\mu_2!} \, \dots \, \frac{x_n^{\mu_n}}{\mu_n!} \ ,
\end{equation}
which will give rise to a Hopf algebra dual form consisting in the product of usual exponentials,
\begin{equation}
    T = e^{x_1 \otimes X_1} \, e^{x_2 \otimes X_2} \, \dots \, e^{x_n \otimes X_n} \ .
\end{equation}
As a consequence, the contraction map $\varphi$ acting on the basis elements of $\mathcal{O}_z(G)$ and being compatible with the Hopf algebra duality is of the form
\begin{equation} \label{dBasisC}
    \varphi_\varepsilon(x_\mu) = \frac{\varphi_\varepsilon(x_1)^{\mu_1}}{\mu_1!} \, \frac{\varphi_\varepsilon(x_2)^{\mu_2}}{\mu_2!} \, \dots \, \frac{\varphi_\varepsilon(x_n)^{\mu_n}}{\mu_n!} \ ,
\end{equation}
so the contracted $T$-matrix will be given by
\begin{equation}
    T' = \lim_{\varepsilon \to 0} e^{\varphi_\varepsilon(x_1) \otimes \phi_\varepsilon(X_1)} \, e^{\varphi_\varepsilon(x_2) \otimes \phi_\varepsilon(X_2)} \, \dots \, e^{\varphi_\varepsilon(x_n) \otimes \phi_\varepsilon(X_n)} \ .
\end{equation}
It is assumed that $\phi_\varepsilon(z) = \varphi_\varepsilon(z)$ is defined by the corresponding LBC to ensure the convergence of both contractions, the quantum algebra contraction and the quantum group contraction, simultaneously.

We remark that, in generic quantum groups, the only possibility beyond Eq.~\eqref{dBasis} is that some of the factorials turn into $q$-factorials $[n]_q!$ with $q = e^{f(z)}$, so that in the Hopf algebra dual form some of the exponentials will turn into $q$-exponentials (see Eq.~\eqref{Tqexp}). In this case, the contraction map \eqref{dBasisC} includes terms of the type $\varphi_\varepsilon([n]_q)!$, with $\varphi_\varepsilon([n]_q)$ given by
\begin{equation}
    \varphi_\varepsilon([n]_q) = [n]_{\varphi_\varepsilon(q)} \ ,
\end{equation}
and $\varphi_\varepsilon(q) = e^{\varphi_\varepsilon(f(z))}$, with $\phi_\varepsilon(z) = \varphi_\varepsilon(z)$ given by Eq.~\eqref{zC}. As a consequence, depending on the form of $\phi_\varepsilon(z)$, the contraction procedure will produce a contracted universal $T$-matrix that either retains $q$-exponentials or transforms $q$-exponentials into ordinary exponentials.


\section{The quantum (1+1) centrally extended Poincaré \texorpdfstring{$T$}{T}-matrix} \label{Sec4}

The aim of this section is to derive a new quantum deformation of the (1+1) centrally extended Poincaré Lie algebra and obtain its Hopf algebra dual form. We will identify in this way the relativistic analogue of the quantum centrally extended Galilei algebra whose universal $T$-matrix was shown to describe QRF transformations in \cite{Ballesteros:2025ypr}.


\subsection{The quantum centrally extended Galilei algebra}

It is well-known that the classification of the quantum deformations of a given Lie algebra reduces to the classification of its compatible Lie bialgebra structures. Additionally, Lie bialgebras are in one-to-one correspondence with (connected and simply connected) Poisson--Lie groups \cite{Drinfeld:1983ky}, and both can be considered the ``semiclassical limit'' of quantum algebras and quantum groups, respectively, that are Hopf algebra duals. This is the basis of the quantum duality principle summarized in universal $T$-matrices. 

The classification of Poisson structures on the (1+1) Galilei Lie group with non-trivial central extension was performed in \cite{Opanowicz:1998zz} (see also \cite{Opanowicz:1999qp, Ballesteros:1999ew}). Among the 26 non-isomorphic classes, only a particular instance ultimately gives rise to the universal $T$-matrix of non-relativistic QRF transformations \cite{Ballesteros:2025ypr}.

The quantum (1+1) centrally extended Galilei algebra arising as the Hopf algebra dual to the Hopf algebra obtained by quantization of the selected Poisson--Lie Galilei group (as a Poisson--Hopf algebra) is characterized by the following deformed commutation relations and coproducts:\footnote{We recall that this quantum algebra, in a different basis, was first introduced in \cite{Bonechi:1992cb, Bonechi:1992ye} as the symmetry algebra underlying magnon excitations in (1+1)-dimensional Heisenberg spin chains, with the quantum deformation parameter $\alpha$ identified with the lattice spacing.}
\begin{equation} \label{EGalilei QA(a)}
    [P_0, P_1] = 0 \ , \quad [K, M] = - \frac{\alpha}{2} M^2 e^{\alpha P_1} \ , \quad [K, P_0] = \frac{e^{\alpha P_1} - 1}{\alpha} \ , \quad [K, P_1] = M e^{\alpha P_1} \ ,
\end{equation}
and
\begin{equation} \label{EGalilei QA(b)}
   \begin{aligned}
        \Delta(M) &= M \otimes 1 + e^{- \alpha P_1} \otimes M \ , \\
        \Delta(P_0) &= P_0 \otimes 1 + 1 \otimes P_0 \ , \\
        \Delta(P_1) &= P_1 \otimes 1 + 1 \otimes P_1 \ , \\
        \Delta(K) &= K \otimes e^{\alpha P_1} + 1 \otimes K \ ,
    \end{aligned}
\end{equation}
respectively, with the rest of commutators vanishing. The associated Lie bialgebra is consequently defined by the (1+1) centrally extended Galilei Lie algebra, with Lie bracket read off from the zeroth order of Eq.~\eqref{EGalilei QA(a)}, namely
\begin{equation} \label{EGalilei LA}
    [P_0, P_1] = 0 \ , \quad [K, P_0] = P_1 \ , \quad [K, P_1] = M \ , \quad [M, \cdot] = 0 \ ,
\end{equation}
together with the cocommutator map extracted from the (skew-symmetric part of the) first order of Eq.~\eqref{EGalilei QA(b)}, namely
\begin{equation} \label{EGalilei LC}
    \begin{aligned}
        \delta(M) &= \alpha M \wedge P_1 \ , \\
        \delta(P_0) &= 0 \ , \\
        \delta(P_1) &= 0 \ , \\
        \delta(K) &= \alpha K \wedge P_1 \ .
    \end{aligned}
\end{equation}
One can check that this cocommutator is not coming from an $r$-matrix, thus corresponding to a non-coboundary Lie bialgebra. One can see this simply by noting that, in this case, coboundary Lie bialgebras must be such that $\delta(M) = 0$.


\subsection{The centrally extended Poincaré Lie bialgebra}

We begin by highlighting that, to the best of our knowledge, the classification of the Lie bialgebra structures for the (1+1) centrally extended Poincaré Lie algebra has not been completely performed yet. We are interested in answering the following question: Can the Galilei Lie bialgebra defined by Eqs.~\eqref{EGalilei LA} and \eqref{EGalilei LC} be obtained by contraction of a Poincaré Lie bialgebra? The answer is positive, as shown in Proposition \ref{Prop}. What is more, as far as fundamental multiparametric LBC are concerned, uniqueness (up to isomorphism) is demonstrated. Let us mention that the fact that this Lie bialgebra exists, and even more, that it is unique, is not guaranteed. Afterwards, this Lie bialgebra will be quantized to a Hopf algebra deformation, namely the quantum algebra. We will finally construct the corresponding Hopf algebra dual form.

The (1+1) Poincaré Lie algebra with trivial central extension (or pseudo-extension \cite{deAzcarraga:1995uw}) $M$ interpreted as the (rest) mass generator is
\begin{equation} \label{EPoincaré LA}
    [P_0, P_1] = 0 \ , \quad [K, P_0] = P_1 \ , \quad [K, P_1] = M + P_0 \ , \quad [M, \cdot] = 0 \ .
\end{equation}
Note that this Lie algebra is isomorphic to the direct sum of the non-centrally extended one with the Abelian Lie subalgebra generated by the central element.

Let us consider the Lie algebra contraction. Analogously to Examples \ref{Ex1} and \ref{Ex2}, the following contraction map $\phi_c$ is introduced now \cite{Aldaya:1985plo}:
\begin{equation} \label{C1}
    \phi_c(M) = c^{-2} M \ , \quad \phi_c(P_0) = P_0 \ , \quad \phi_c(P_1) = c^{-1} P_1 \ , \quad \phi_c(K) = c^{-1} K \ .
\end{equation}
In this case, Eq.~\eqref{Lim1} reads
\begin{equation}
    [P_0, P_1]' = \lim_{c \to + \infty} 0 \ , \quad [K, P_0]' = \lim_{c \to + \infty} P_1 \ , \quad [K, P_1]' = \lim_{c \to + \infty}(M + c^{-2} P_0) \ , \quad [M, \cdot]' = \lim_{c \to + \infty} 0 \ ,
\end{equation}
and taking the limit $c \to + \infty$,
\begin{equation}
    [P_0, P_1]' = 0 \ , \quad [K, P_0]' = P_1 \ , \quad [K, P_1]' = M \ , \quad [M, \cdot]' = 0 \ ,
\end{equation}
which is the (1+1) centrally extended Galilei Lie algebra \eqref{EGalilei LA}. Note that this Lie algebra is not isomorphic to the direct sum of the non-centrally extended one with Abelian Lie subalgebra generated by the central element.

\begin{proposition} \label{Prop}
    The unique Lie bialgebra structure for the (1+1) centrally extended Poincaré Lie algebra \eqref{EPoincaré LA} such that its fundamental Lie bialgebra contraction gives Eq.~\eqref{EGalilei LC} is
    \begin{equation} \label{EPoincaré LC}
        \begin{aligned}
            \delta(M) &= \alpha M \wedge P_1 + \alpha P_0 \wedge P_1 \ , \\
            \delta(P_0) &= 0 \ , \\
            \delta(P_1) &= 0 \ , \\
            \delta(K) &= \alpha K \wedge P_1 \ .
        \end{aligned}
    \end{equation}
\end{proposition}

\begin{proof}
    We begin by considering a generic element on $\mathfrak{g} \wedge \mathfrak{g}$,
    \begin{equation}
        \delta(X^i) = f^i_{j k} X^j \wedge X^k \ ,
    \end{equation}
    with $\mathfrak{g}$ defined by Eq.~\eqref{EPoincaré LA}. If we apply the 1-cocycle condition \eqref{1-cocycle}, we obtain the following (pre-)cocommutator map:
    \begin{equation} \label{EPoincaré LC(p)}
        \begin{aligned}
            \delta(M) &= \alpha M \wedge P_1 + \alpha P_0 \wedge P_1 \ , \\
            \delta(P_0) &= \beta_1 M \wedge P_0 + \beta_2 M \wedge P_1 + \beta_3 M \wedge K + \beta_4 P_0 \wedge P_1 + \beta_3 P_0 \wedge K \ , \\
            \delta(P_1) &= (\beta_2 - \beta_4) M \wedge P_0 + \beta_5 M \wedge P_1 - (\beta_1 - \beta_5) P_0 \wedge P_1 + \beta_3 P_1 \wedge K \ , \\
            \delta(K) &= \beta_6 M \wedge P_0 + \beta_7 M \wedge P_1 - (\beta_1 - \beta_5) M \wedge K + \beta_8 P_0 \wedge P_1 - (\beta_1 - \beta_5) P_0 \wedge K \\
            & \quad - (\alpha + \beta_4) P_1 \wedge K \ ,
        \end{aligned}
    \end{equation}
    in terms of 9 parameters. Afterwards, the condition that $(\mathfrak{g^*}, \delta^*)$ defines a Lie algebra provides the following 6 quadratic constraints for these parameters:
    \begin{equation} \label{c}
        \begin{gathered}
            \alpha \beta_1 = 0 \ , \quad \alpha \beta_3 = 0 \ , \quad \beta_1 \beta_8 = 0 \ , \quad \beta_3 \beta_8 = 0 \ , \\
            \beta_1 \beta_2 - (\beta_2 - \beta_4) \beta_5 + \beta_3 \beta_7 = 0 \ , \\
            (\alpha + \beta_4) (\beta_2 - \beta_4) + \beta_1 (\beta_1 - \beta_5) + \beta_3 \beta_6 = 0 \ .
        \end{gathered}
    \end{equation}
    We aim at solutions such that their contraction produces Eq.~\eqref{EGalilei LC}, therefore solutions such that $\alpha \neq 0$. Accordingly, Eq.~\eqref{c} reduces to
    \begin{equation}
        \alpha \neq 0 \ , \quad \beta_1 = 0 \ , \quad \beta_3 = 0 \ , \quad (\beta_2 - \beta_4) \beta_5 = 0 \ , \quad (\alpha + \beta_4) (\beta_2 - \beta_4) = 0 \ .
    \end{equation}
    The solutions are
    \begin{itemize}
        \item[i)] $\beta_4 = - \alpha$, $\beta_5 = 0$,
        \item[ii)] $\beta_4 = \beta_2$.
    \end{itemize}
    It is clear that obtaining Eq.~\eqref{EGalilei LC} from Eq.~\eqref{EPoincaré LC(p)} is not possible with solution i), since the term $P_1 \wedge K$ vanishes. As a consequence, the only possible solution is ii). The associated cocommutator reads
    \begin{equation} \label{EPoincaré LC(p)2}
        \begin{aligned}
            \delta(M) &= \alpha (M \wedge P_1 + P_0 \wedge P_1) \ , \\
            \delta(P_0) &= \beta_2 (M \wedge P_1 + P_0 \wedge P_1) \ , \\
            \delta(P_1) &= \beta_5 (M \wedge P_1 + P_0 \wedge P_1) \ , \\
            \delta(K) &= \beta_6 M \wedge P_0 + \beta_7 M \wedge P_1 + \beta_8 P_0 \wedge P_1 + \beta_5 (M \wedge K + P_0 \wedge K) - (\alpha + \beta_2) P_1 \wedge K \ .
        \end{aligned}
    \end{equation}
    If we apply the contraction map \eqref{C1}, we have that Eq.~\eqref{Lim3} reads 
    \begin{equation}
        \begin{aligned}
            \delta'(M) &= \lim_{c \to + \infty} \alpha c (M \wedge P_1 + c^{-2} P_0 \wedge P_1) \ , \\
            \delta'(P_0) &= \lim_{c \to + \infty} \beta_2 c^3 (M \wedge P_1 + c^{-2} P_0 \wedge P_1) \ , \\
            \delta'(P_1) &= \lim_{c \to + \infty} \beta_5 c^2 (M \wedge P_1 + c^{-2} P_0 \wedge P_1) \ , \\
            \delta'(K) &= \lim_{c \to + \infty} (\beta_6 c M \wedge P_0 + \beta_7 c^2 M \wedge P_1 + \beta_8 P_0 \wedge P_1 + \beta_5 c^2 (M \wedge K + c^{-2} P_0 \wedge K) \\
            & \qquad \qquad - (\alpha + \beta_2) c P_1 \wedge K) \ .
        \end{aligned}
    \end{equation}
    We see that to have a well-behaved cocommutator under the contraction limit $c \to + \infty$, the deformation parameters $\alpha$, $\beta_2$, $\beta_5$, $\beta_6$, $\beta_7$, $\beta_8$ must transform as  $\alpha' = \phi_c(\alpha) = \alpha c$, $\beta_2' = \phi_c(\beta_2) = \beta_2 c^3$, $\beta_5' = \phi_c(\beta_5) = \beta_5 c^2$, $\beta_6' = \phi_c(\beta_6) = \beta_6 c$, $\beta_7' = \phi_c(\beta_7) = \beta_7 c^2$, $\beta_8' = \phi_c(\beta_8) = \beta_8$, respectively (see Eq.~\eqref{zC}). This defines the fundamental multiparametric Lie bialgebra contraction. The result is
    \begin{equation}
        \begin{aligned}
            \delta'(M) &= \alpha' M \wedge P_1 \ , \\
            \delta'(P_0) &= \beta_2' M \wedge P_1  \ , \\
            \delta'(P_1) &= \beta_5' M \wedge P_1  \ , \\
            \delta'(K) &= \beta_6' M \wedge P_0 + \beta_7' M \wedge P_1 + \beta_8' P_0 \wedge P_1 + \beta_5' M \wedge K - \alpha' P_1 \wedge K \ .
        \end{aligned}
    \end{equation}
    This cocommutator coincides with Eq.~\eqref{EGalilei LC} if
    \begin{equation}
        \beta_2' = \beta_5' = \beta_6' = \beta_7' = \beta_8' = 0 \ ,
    \end{equation}
    then, under the condition of fundamental multiparametric Lie bialgebra contraction, 
    \begin{equation}
        \beta_2 = \beta_5 = \beta_6 = \beta_7 = \beta_8 = 0 \ .
    \end{equation}
    therefore, recalling Eq.~\eqref{EPoincaré LC(p)2},
    \begin{equation}
        \begin{aligned}
            \delta(M) &= \alpha M \wedge P_1 + \alpha P_0 \wedge P_1 \ , \\
            \delta(P_0) &= 0 \ , \\
            \delta(P_1) &= 0 \ , \\
            \delta(K) &= \alpha K \wedge P_1 \ ,
        \end{aligned}
    \end{equation}
    thus Eq.~\eqref{EPoincaré LC} is the unique Lie bialgebra structure for the (1+1) centrally extended Poincaré Lie algebra \eqref{EPoincaré LA} (up to isomorphism) whose fundamental Lie bialgebra contraction recovers Eq.~\eqref{EGalilei LC}. This cocommutator will again define a non-coboundary Lie bialgebra.
\end{proof}

We will consider Hopf algebra dual forms for quantum groups, motivated by their role as QRF transformations in this work. Additionally, such universal $T$-matrices are natural objects to act on noncommutative spaces and describe their symmetries. As already mentioned, a quantum group can be considered as the quantization of a Poisson--Lie group. We recall that a Poisson--Lie group $(G, \Pi_G)$ is a Lie group $G$ equipped with a Poisson structure $\Pi_G$, making $G$ into a Poisson manifold, such that $\Pi_G$ is multiplicative. In the simplest case, this Poisson-Lie group $(G, \Pi_G)$ describes the symmetries of the coisotropic Poisson quotient $(M = G/H, \pi = \Pi_{G/H})$, where $H$ is a closed subgroup of $G$, the isotropy subgroup of a point $m \in M$, and where the canonical projection $p: G \rightarrow M$ is a Poisson epimorphism. This last condition, at the infinitesimal level, is the coisotropy condition $\delta(\mathfrak{h}) \subset \mathfrak{h} \wedge \mathfrak{g}$. The quantization of coisotropic Poisson quotients gives rise to embeddable noncommutative homogeneous spaces.

In this context, the (1+1) centrally extended Galilei and Poincaré Lie bialgebras given in Eqs.~\eqref{EGalilei LC} and \eqref{EPoincaré LC} are distinct. We have that, for both of them, $\delta(\mathfrak{l}) \subset \mathfrak{l} \wedge \mathfrak{g}$ with $\mathfrak{l} = \langle K \rangle$. We can then construct the coisotropic Poisson quotient $(G/L, \Pi_{G/L})$, with local coordinates $\{\theta, a_0, a_1\}$. In the Galilei case, the Poisson structure is given by 
\begin{equation}
    \{a_0, a_1\} = 0 \ , \quad \{\theta, a_0\} = 0 \ , \quad \{\theta, a_1\} = \alpha \theta \ .
\end{equation}
In the Poincaré case, the Poisson structure is given by
\begin{equation}
        \{a_0, a_1\} = \alpha \theta \ , \quad \{\theta, a_0\} = 0 \ , \quad \{\theta, a_1\} = \alpha \theta \ .
\end{equation}
Additionally, only for the Galilei Lie bialgebra structure \eqref{EGalilei LC}, $\delta(\bar{\mathfrak{l}}) \subset \bar{\mathfrak{l}} \wedge \mathfrak{g}$ with $\bar{\mathfrak{l}} = \langle M, K \rangle$. In the coisotropic Poisson quotient $(G/\bar{L}, \Pi_{G/\bar{L}})$, with local coordinates $\{a_0, a_1\}$, the Poisson structure is trivial:
\begin{equation}
    \{a_0, a_1\} = 0 \ .
\end{equation}
These coisotropic Poisson spaces quantize to noncommutative spaces with commutators embedded in the commutation relations of the quantum group coordinates (see Eqs. \eqref{EGalilei QG(a)} and \eqref{EPoincaré QG(a)}) in the sense that the inclusion is an algebra monomorphism into the underlying algebra structure of the corresponding Hopf algebra, namely the inclusion is an injective algebra homomorphism.

It is also interesting to realize that, in the classification of the Lie bialgebra structures for the (1+1) centrally extended Galilei Lie algebra performed in \cite{Ballesteros:1999ew}, Eq.~\eqref{EGalilei LC} precisely corresponds to the Family II(b) representative after applying the coisotropy condition with respect to $\mathfrak{l}$. We will comment later on the corresponding Galilei and Poincaré embeddable noncommutative homogeneous spaces obtained as quantizations of the coisotropic Poisson quotients of the type $(G/L, \Pi_{G/L})$. We will argue that, in a broad sense, they can generally be interpreted as noncommutative principal bundles over quantum spacetimes.


\subsection{The quantum centrally extended Poincaré algebra}

The next step is the quantization of this Lie bialgebra $(\mathfrak{g}, \delta)$, which yields a quantum algebra $U_\alpha(\mathfrak{g})$, namely a Hopf algebra deformation of the universal enveloping algebra $U(\mathfrak{g})$ such that the Lie bracket \eqref{EPoincaré LA} gives the zeroth order of the commutations relations in the deformation parameter, \textit{i.e.}
\begin{equation}
    [\cdot, \cdot] = [\cdot, \cdot]_0 + \mathcal{O}(\alpha) \ ,
\end{equation}
with $[\cdot, \cdot]_0$ the Lie bracket of $\mathfrak{g}$ extended to $U(\mathfrak{g})$, and the cocommutator \eqref{EPoincaré LC} gives (the skew-symmetric part of) the first order of the coproduct in the deformation parameter, \textit{i.e.}
\begin{equation}
    \Delta = \Delta_0 + \Delta_1 + \mathcal{O}(\alpha^2) \ ,
\end{equation}
with $\Delta_0$ the primitive coproduct and $\delta = \frac{1}{2} (\Delta_1 - \tau \circ \Delta_1)$. We recall that this relation is one-to-one. This commutation relations and this coproduct must be compatible, namely such that the homomorphism condition between the coproduct and the product is verified, \textit{i.e.}~$\Delta([X, Y]) = [\Delta(X), \Delta(Y)], \, \forall X, Y \in U_\alpha(\mathfrak{g})$. 

In general, this a hard problem to solve. We can take advantage in this case of the computational approach exploited in \cite{Ballesteros2013} to obtain quantum algebras as quantizations of dual Poisson--Lie groups. We recall that, after deriving the dual Poisson--Lie group element whose Lie algebra is the dual of Eq.~\eqref{EPoincaré LC}, the coproduct of the quantum Poincaré algebra of interest is obtained using the group multiplication for the local coordinates of the dual Poisson--Lie group. Afterwards, the lengthy but straightforward computation described in \cite{Ballesteros2013} leads to the unique Poisson--Lie structure on the dual Poisson--Lie group whose linearization is isomorphic to the (1+1) centrally extended Poincaré Lie algebra \eqref{EPoincaré LA}. Finally, this Poisson bracket is quantized straightforwardly (no ordering problems show up), thus providing the commutation relations of the quantum Poincaré algebra in question.

We present the result in the following proposition:
\begin{proposition} \label{EPoincaré QA}
    This quantum (1+1) centrally extended Poincaré algebra has the following deformed coproducts for its generators:
    \begin{equation} \label{EPoincaré QA(b)}
        \begin{aligned}
            \Delta(M) &= M \otimes 1 + e^{- \alpha P_1} \otimes M + (e^{- \alpha P_1} - 1) \otimes P_0 \ , \\
            \Delta(P_0) &= P_0 \otimes 1 + 1 \otimes P_0 \ , \\
            \Delta(P_1) &= P_1 \otimes 1 + 1 \otimes P_1 \ , \\
            \Delta(K) &= K \otimes e^{\alpha P_1} + 1 \otimes K \ ,
        \end{aligned}
    \end{equation}
    with the following compatible deformed commutation relations:
    \begin{equation} \label{EPoincaré QA(a)}
        \begin{gathered}
            \relax [P_0, P_1] = 0 \ , \quad [M, P_0] = 0 \ , \quad [M, P_1] = 0 \ , \\
            [K, P_0] = \frac{e^{\alpha P_1} - 1}{\alpha} \ , \quad [K, P_1] = (M + P_0) e^{\alpha P_1} \ , \\
            [K, M] = - \left(\frac{\cosh{(\alpha P_1)} - 1}{\alpha} + \frac{\alpha}{2} (M + P_0)^2 e^{\alpha P_1}\right) \ .
        \end{gathered}
    \end{equation}
\end{proposition}
Notice that $M$ is not anymore central in this particular quantum deformation. One can see that, in the limit $\alpha \to 0$, the Poincaré Lie algebra \eqref{EPoincaré LA} is recovered. We have thus obtained the quantum (1+1) Poincaré algebra $U_\alpha(\mathfrak{g})$. This Hopf algebra deformation is characterized by the commutation relations and coproducts in Eqs.~\eqref{EPoincaré QA(a)} and \eqref{EPoincaré QA(b)}.


\subsection{The universal \texorpdfstring{$T$}{T}-matrix and the quantum Poincaré group}

We finally arrive at the main point of this section, namely the construction of the universal $T$-matrix for this quantum Poincaré algebra. The basics of this procedure were outlined in Section \ref{Sec3} for the case of the $\kappa$-Poincaré quantum algebra and calculations for this quantum Poincaré algebra are described in detail in the \nameref{Appendix}.

The result is summarized in the following proposition:
\begin{proposition}
    The Hopf algebra dual form \eqref{T0} for the quantum (1+1) centrally extended Poincaré algebra obtained in this section is
    \begin{equation} \label{EPoincaré T}
        T = e^{\theta \otimes M} e^{a_0 \otimes P_0} e^{a_1 \otimes P_1} e^{\chi \otimes K} \ ,
    \end{equation}
    where $\theta$, $a_0$, $a_1$, and $\chi$ are the quantum group coordinates, dually related to the accompanying quantum algebra generators in the exponentials of the universal $T$-matrix.
\end{proposition}
Its dual nature allows us to derive the commutation relations and the coproducts of the local coordinates in the same way as done in Section \ref{Sec3} for the $\kappa$-Poincaré quantum algebra. We also provide some steps of the computation for this quantum Poincaré algebra in the \nameref{Appendix}.

We present the result in the following proposition:
\begin{proposition} \label{EPoincaré QG}
    The quantum (1+1) centrally extended Poincaré group has the following coproducts for its local coordinates:
    \begin{equation} \label{EPoincaré QG(b)}
        \begin{aligned}
            \Delta(\theta) &= \theta \otimes 1 + 1 \otimes \theta + (\cosh\chi - 1) \otimes a_0 + \sinh\chi \otimes a_1 \ , \\
            \Delta(a_0) &= a_0 \otimes 1 + \cosh\chi \otimes a_0 + \sinh\chi \otimes a_1 \ , \\
            \Delta(a_1) &= a_1 \otimes 1 + \cosh\chi \otimes a_1 + \sinh\chi \otimes a_0 \ , \\
            \Delta(\chi) &= \chi \otimes 1 + 1 \otimes \chi \ ,
        \end{aligned}
    \end{equation}
    with the following compatible deformed commutation relations:
    \begin{equation} \label{EPoincaré QG(a)}
        \begin{gathered}
            \relax [a_0, a_1] = \alpha \theta \ , \quad [\theta, a_0] = 0 \ , \quad [\theta, a_1] = \alpha \theta \ , \quad \\
            [\chi, \theta] = \alpha (\cosh\chi - 1) \ , \quad [\chi, a_0] = \alpha (\cosh\chi - 1) \ , \quad [\chi, a_1] = \alpha \sinh\chi \ .
        \end{gathered}
    \end{equation}
\end{proposition}
The Hopf algebra dual form \eqref{EPoincaré T} is intended to be the relativistic analogue of the universal $T$-matrix of non-relativistic QRF transformations.

What is also worth highlighting is that the embeddable noncommutative homogeneous Minkowski spacetime is of the form $[a_0, a_1] = \alpha \theta$. Notice that this quantum homogeneous spacetime can be considered a generalization of the well-known canonical or Moyal noncommutative spacetime \cite{Groenewold:1946kp, Moyal:1949sk, Doplicher:1994zv, Doplicher:1994tu, Douglas:2001ba, Szabo:2001kg} where $\theta$ is no longer a central extension of the coordinate algebra.

What is more, interestingly, when a (finite-dimensional) Lie algebra $\mathfrak{g}$ is centrally extended to $\mathfrak{\bar{g}}$ and $\mathfrak{\bar{g}}$ integrates to a (finite-dimensional) Lie group with central extension $\bar{G}$, then $\bar{G}$ turns into a (topologically trivial) principal bundle over the Lie group $G$ with the central extension defining the structure group. This is always the case regardless of the nature of the central extension, \textit{i.e.}~trivial or non-trivial. In the case of a trivial central extension, the principal bundle is also trivial as a centrally extended Lie group. The reader is referred to \cite{Tuynman:1987ij, Neeb2002} and references therein for a thorough mathematical treatment of central extensions of Lie algebras and Lie groups. In the cases of interest, \textit{i.e.}~the (1+1) centrally extended Galilei Lie algebra \eqref{EGalilei LA} and the (1+1) centrally extended Poincaré Lie algebra \eqref{EPoincaré LA}, the central extension also integrates to the Lie groups and such Lie groups with central extension are topologically trivial principal $\mathbb{R}$-bundles (or $U(1)$ by quotienting by a discrete subgroup). In contrast with the Galilei case, in the Poincaré case the principal bundle is also trivial as a Lie group with central extension.

One can then construct homogeneous spaces of centrally extended Lie groups as $\bar{G}/\bar{H}$, with $\bar{H}$ a closed subgroup of $\bar{G}$ (not only of $G$). In the appropriate cases, consequently, such homogeneous extended spaces are principal (gauge) bundles over spacetimes. This is the case, for example, of the homogeneous extended Minkowski space obtained by quotienting the (1+1) centrally extended Poincaré Lie group by the Lorentz subgroup, which is a principal $U(1)$-bundle over (1+1) Minkowski spacetime. Note that the coordinate on the $U(1)$ fibers is the coordinate dual to the Lie algebra central extension. As a consequence, in the context of \cite{Ballesteros:2018ghw} where noncommutative principal bundles are understood as noncommutative algebras of functions on principal bundles\footnote{Note that this notion is based on considering only the manifold underlying a principal bundle. In noncommutative geometry, the complete definition of quantum principal bundle imposes additional conditions and is consequently more restrictive \cite{Brzezinski:1992hda, Brzezinski:1992ut, Budzynski:1993dm, Pflaum:1993mt}.}, the embeddable noncommutative homogeneous extended Minkowski space can be considered as a noncommutative principal $U(1)$-bundle over a quantum (1+1) Minkowski spacetime, essentially given by the first line of Eq.~\eqref{EPoincaré QG(a)}. Notice that the (would-be) base spacetime coordinates $a_0$ and $a_1$ do not commute and its commutator is proportional to the coordinate on the $U(1)$ fibers $\theta$, which itself commutes with the time coordinate $a_0$ and does not commute with the space coordinate $a_1$.

One can choose an appropriate basis in this case where the (1+1) centrally extended Poincaré Lie algebra is explicitly expressed as the direct sum of the non-centrally extended one with the central extension. In this basis, the classification of the embeddable noncommutative homogeneous extended Minkowski spaces at first order was performed in \cite{Ballesteros:2018ghw}.

Let us look further at the (1+1) Poincaré Lie algebra with trivial central extension given in Eq.~\eqref{EPoincaré LA}. The aforementioned basis change is just
\begin{equation} \label{BC}
    \tilde{M} = M \ , \quad \tilde{P}_0 = M + P_0 \ , \quad \tilde{P}_1 = P_1 \ , \quad \tilde{K} = K \ .
\end{equation}
The Lie algebra reads simply\footnote{This Lie algebra is also known as the Nappi--Witten algebra and it plays a relevant role in (1+1) gravity \cite{Cangemi:1993sd, Nappi:1993ie}.}
\begin{equation} \label{EPoincaré LA*}
    [\tilde{P}_0, \tilde{P}_1] = 0 \ , \quad [\tilde{K}, \tilde{P}_0] = \tilde{P}_1 \ , \quad [\tilde{K}, \tilde{P}_1] = \tilde{P}_0 \ , \quad [\tilde{M}, \cdot] = 0 \ .
\end{equation}
At the Lie bialgebra level, the cocommutator \eqref{EPoincaré LC} now reads
\begin{equation} \label{EPoincaré LC*}
    \begin{aligned}
        \delta(\tilde{M}) &= \alpha \tilde{P}_0 \wedge \tilde{P}_1 \ , \\
        \delta(\tilde{P}_0) &= \alpha \tilde{P}_0 \wedge \tilde{P}_1 \ , \\
        \delta(\tilde{P}_1) &= 0 \ , \\
        \delta(\tilde{K}) &= \alpha \tilde{K} \wedge \tilde{P}_1 \ .
    \end{aligned}
\end{equation}
One can check in the classification of \cite{Ballesteros:2018ghw} that this Lie bialgebra corresponds to the case B.2, with cocommutator reproduced here with slightly different notation:
\begin{equation}
    \begin{aligned}
        \delta(\tilde{M}) &= \gamma \tilde{P}_0 \wedge \tilde{P}_1 \ , \\
        \delta(\tilde{P}_0) &= v_1 \tilde{P}_0 \wedge \tilde{P}_1 \ , \\
        \delta(\tilde{P}_1) &= - v_0 \tilde{P}_0 \wedge \tilde{P}_1 \ , \\
        \delta(\tilde{K}) &= v_0 \tilde{K} \wedge \tilde{P}_0 + v_1 \tilde{K} \wedge \tilde{P}_1 \ .
    \end{aligned}
\end{equation}
We immediately see that Eq.~\eqref{EPoincaré LC*} is the instance with $v_0 = 0$ and $v_1 = \gamma = \alpha$. In the new basis, the coproducts and the commutation relations of the quantum algebra of Proposition \ref{EPoincaré QA} now become
\begin{equation}
    \begin{aligned}
        \Delta(\tilde{M}) &= \tilde{M} \otimes 1 + 1 \otimes \tilde{M} + (e^{- \alpha \tilde{P}_1} - 1) \otimes \tilde{P}_0 \ , \\
        \Delta(\tilde{P}_0) &= \tilde{P}_0 \otimes 1 + e^{- \alpha \tilde{P}_1} \otimes \tilde{P}_0 \ , \\
        \Delta(\tilde{P}_1) &= \tilde{P}_1 \otimes 1 + 1 \otimes \tilde{P}_1 \ , \\
        \Delta(\tilde{K}) &= \tilde{K} \otimes e^{\alpha \tilde{P}_1} + 1 \otimes \tilde{K} \ ,
    \end{aligned}
\end{equation}
and
\begin{equation}
    \begin{gathered}
        \relax [\tilde{P}_0, \tilde{P}_1] = 0 \ , \quad [\tilde{M}, \tilde{P}_0] = 0 \ , \quad [\tilde{M}, \tilde{P}_1] = 0 \ , \\
        [\tilde{K}, \tilde{P}_0] = \frac{\sinh{(\alpha \tilde{P}_1)}}{\alpha} - \frac{\alpha}{2} \tilde{P}_0^2 e^{\alpha \tilde{P}_1} \ , \quad [\tilde{K}, \tilde{P}_1] = \tilde{P}_0 e^{\alpha \tilde{P}_1} \ , \\
        [\tilde{K}, \tilde{M}] = - \left(\frac{\cosh{(\alpha \tilde{P}_1)} - 1}{\alpha} + \frac{\alpha}{2} \tilde{P}_0^2 e^{\alpha \tilde{P}_1}\right) \ ,
    \end{gathered}
\end{equation}
respectively. One can show that the change of basis \eqref{BC} in a Lie bialgebra, induces the change of basis in the dual Lie bialgebra given by
\begin{equation}
    \tilde{\theta} = \theta - a_0 \ , \quad \tilde{a}_0 = a_0 \ , \quad \tilde{a}_1 = a_1 \ , \quad \tilde{\chi} = \chi \ .
\end{equation}
In the new coordinates, the coproducts and commutation relations of the quantum group of Proposition \ref{EPoincaré QG} now become
\begin{equation}
    \begin{aligned}
        \Delta(\tilde{\theta}) &= \tilde{\theta} \otimes 1 + 1 \otimes \tilde{\theta} \ , \\
        \Delta(\tilde{a}_0) &= \tilde{a}_0 \otimes 1 + \cosh\tilde{\chi} \otimes \tilde{a}_0 + \sinh\tilde{\chi} \otimes \tilde{a}_1 \ , \\
        \Delta(\tilde{a}_1) &= \tilde{a}_1 \otimes 1 + \cosh\tilde{\chi} \otimes \tilde{a}_1 + \sinh\tilde{\chi} \otimes \tilde{a}_0 \ , \\
        \Delta(\tilde{\chi}) &= \tilde{\chi} \otimes 1 + 1 \otimes \tilde{\chi} \ ,
    \end{aligned}
\end{equation}
and
\begin{equation}
    \begin{gathered}
        \relax [\tilde{a}_0, \tilde{a}_1] = \alpha (\tilde{\theta} + \tilde{a}_0) \ , \quad [\tilde{\theta}, \tilde{a}_0] = 0 \ , \quad [\tilde{\theta}, \tilde{a}_1] = 0 \ , \\
        [\tilde{\chi}, \tilde{\theta}] = 0 \ , \quad [\tilde{\chi}, \tilde{a}_0] = \alpha (\cosh\tilde{\chi} - 1) \ , \quad [\tilde{\chi}, \tilde{a}_1] = \alpha \sinh\tilde{\chi} \ ,
    \end{gathered}
\end{equation}
respectively. Note that this is manifestly a Hopf algebra centrally extended non-trivially, namely the Hopf algebra cannot be expressed as the tensor product of the non-centrally extended one with the cocommutative and commutative Hopf subalgebra generated by the central element. It corresponds, more precisely, to a non-trivial central extension of the (1+1) spacelike $\kappa$-Poincaré quantum group, with central extension $\tilde{\theta}$.\footnote{To compare with the (1+1) timelike $\kappa$-Poincaré quantum group considered in Section \ref{Sec2}. We recall that the spacelike counterpart follows immediately by interchanging the role of space and time coordinates.} As a consequence, the embeddable noncommutative homogeneous Minkowski spacetime indeed is a (trivially as a Lie algebra, non-trivially as a comodule algebra) centrally extended (1+1) spacelike $\kappa$-Minkowski noncommutative spacetime, in these coordinates of the form $[\tilde{a}_0, \tilde{a}_1] = \alpha (\tilde{\theta} + \tilde{a}_0)$. In other words, this quantum homogeneous spacetime is a combination of the Moyal plus the spacelike $\kappa$-Minkowski noncommutative spacetimes in 1+1 dimensions. It becomes, in the context of \cite{Ballesteros:2018ghw}, a commutative principal $U(1)$-bundle over the spacelike (1+1) $\kappa$-Minkowski spacetime, where the coordinate on the $U(1)$ fibers appears as a central extension of the base $\kappa$-Minkowski spacetime commutator.

Let us close this section stressing that the adequate basis to perform the contraction procedure corresponding the non-relativistic limit is not the basis explicitly showing that the central extension of the Poincaré Lie algebra is trivial, since in this basis the contraction does not produce the Galilei Lie algebra with non-trivial central extension, but with trivial one. This is why the (1+1) centrally extended Poincaré Lie algebra was initially written as in Eq.~\eqref{EPoincaré LA}. That is the basis to be considered for contractions to the non-relativistic scenario, as demonstrated explicitly in this section. The other basis is very convenient for other purposes. It is adapted to the Lie algebra's direct-sum decomposition and, moreover, its dual basis enables the explicit identification of the dual Hopf algebra with a non-trivial central extension of the (1+1) spacelike $\kappa$-Poincaré quantum group.


\section{The Galilei contraction and the \texorpdfstring{$T$}{T}-matrix for QRFs} \label{Sec5}

We consider here the theory of contractions revised in Section \ref{Sec3}, including the newly introduced contractions of the Hopf algebra dual form, to obtain the Galilei $T$-matrix of QRF transformations. We already applied the contraction scheme at the level of the (1+1) centrally extended Poincaré Lie algebra \eqref{EPoincaré LA} and Lie bialgebra \eqref{EPoincaré LC} in Section \ref{Sec4}. We complete the picture devoting this section to the contractions at the deformed Hopf algebra level. We will finish with the contraction of the corresponding Hopf algebra dual form, ultimately recovering the sought non-relativistic universal $T$-matrix.

We recall from Section \ref{Sec3} that Lie bialgebra contraction uniquely defines the contraction of the corresponding quantum algebra, and of the corresponding quantum group through the contraction of the dual Lie bialgebra. The quantum Poincaré algebra obtained in Section \ref{Sec4} is defined by the commutation relations \eqref{EPoincaré QA(a)} and coproducts \eqref{EPoincaré QA(b)}. The contraction procedure renders
\begin{equation}
    [P_0, P_1]' = 0 \ , \quad [K, M]' = - \frac{\alpha'}{2} M^2 e^{\alpha' P_1} \ , \quad [K, P_0]' = \frac{e^{\alpha' P_1} - 1}{\alpha'} \ , \quad [K, P_1]' = M e^{\alpha' P_1} \ ,
\end{equation}
for the commutation relations (see Eq.~\eqref{EGalilei QA(a)}), and
\begin{equation}
   \begin{aligned}
        \Delta'(M) &= M \otimes 1 + e^{- \alpha' P_1} \otimes M \ , \\
        \Delta'(P_0) &= P_0 \otimes 1 + 1 \otimes P_0 \ , \\
        \Delta'(P_1) &= P_1 \otimes 1 + 1 \otimes P_1 \ , \\
        \Delta'(K) &= K \otimes e^{\alpha' P_1} + 1 \otimes K \ ,
    \end{aligned}
\end{equation}
for the coproducts (see Eq.~\eqref{EGalilei QA(b)}), thus obtaining the quantum Galilei algebra of non-relativistic QRF transformations. Notice that the difference between the Galilei and Poincaré coproducts is the second term in $\Delta(M)$ in \eqref{EPoincaré QA(b)}. 

The contraction map $\varphi_c$ for the quantum group coordinates is given by duality with respect to Eq.~\eqref{C1}, hence
\begin{equation} \label{C2}
    \varphi_c(\theta) = c^2 \theta \ , \quad \varphi_c(a_0) = a_0 \ , \quad \varphi_c(a_1) = c a_1 \ , \quad \varphi_c(\chi) = c \chi \ ,
\end{equation}
and with the same transformation for the deformation parameter $\alpha$, namely $\alpha' = \varphi_c (\alpha) = \alpha c$. The quantum Poincaré group obtained in Section \ref{Sec4}, dual to the quantum Poincaré algebra, is defined by the commutation relations \eqref{EPoincaré QG(a)} and coproducts \eqref{EPoincaré QG(b)}. An analogous contraction procedure with the contraction map \eqref{C2} yields\footnote{We conveniently relabel the coordinates as $\theta \rightarrow \theta$, $a_0 \rightarrow b$, $a_1 \rightarrow a$, $\chi \rightarrow v$ for a better comparison with Ref. \cite{Ballesteros:2025ypr}, since this is the notation used for the (1+1) centrally extended Galilei Lie group coordinates.\label{fn}}
\begin{equation} \label{EGalilei QG(a)}
    [a_0, a_1]' = 0 \ , \quad [\theta, a_1]' = \alpha' \theta \ , \quad [\theta, v]' = - \frac{1}{2} \alpha' v^2 \ , \quad [a_1, v]' = - \alpha' v \ ,
\end{equation}
for the commutation relations (with the rest zero), and
\begin{equation} \label{EGalilei QG(b)}
    \begin{aligned}
        \Delta'(\theta) &= \theta \otimes 1 + 1 \otimes \theta + v \otimes a + \frac{1}{2} v^2 \otimes b \ , \\
        \Delta'(b) &= b \otimes 1 + 1 \otimes b \ , \\
        \Delta'(a) &= a \otimes 1 + 1 \otimes a + v \otimes b \ , \\
        \Delta'(v) &= v \otimes 1 + 1 \otimes v \ ,
    \end{aligned}
\end{equation}
for the coproducts, thus obtaining the quantum Galilei group of non-relativistic QRF transformations.

The last step is to apply the contraction theory of Hopf algebra dual forms to the Poincaré $T$-matrix \eqref{EPoincaré T}. Finally, Eq.~\eqref{Lim4} produces
\begin{equation}
    T' = \lim_{c \to + \infty} e^{\varphi_c(\theta) \otimes \phi_c(M)} \, e^{\varphi_c(a_0) \otimes \phi_c(P_0)} \, e^{\varphi_c(a_1) \otimes \phi_c(P_1)} \, e^{\varphi_c(\chi) \otimes \phi_c(K)} \ ,
\end{equation}
and taking the limit $c \to + \infty$,
\begin{equation} \label{EGalilei T}
    T' = e^{\theta \otimes M} \, e^{b \otimes P_0} \, e^{a \otimes P_1} \, e^{v \otimes K} \ ,
\end{equation}
which is the Galilei $T$-matrix formalizing the framework of non-relativistic QRF transformations. As already anticipated at the end of the previous section, the contracted Galilei Hopf algebra dual form \eqref{EGalilei T} is formally identical to the initial Poincaré one \eqref{EPoincaré T}, modulo coordinate relabeling (see Footnote \ref{fn}).

Let us mention that, remarkably, the embeddable noncommutative homogeneous extended Galilei space renders, in the sense of \cite{Ballesteros:2018ghw}, a noncommutative principal $U(1)$-bundle over a classical (1+1) Galilei spacetime. One can check in Eq.~\eqref{EGalilei QG(a)} that the (would-be) base spacetime becomes commutative but the coordinate on the $U(1)$ fibers inherits the same noncommutativity with spacetime coordinates as the corresponding Poincaré counterpart.


\section{Concluding remarks} \label{Sec6}

This article is based on the connection, recently reported in \cite{Ballesteros:2025ypr}, between the QRF transformations introduced in \cite{Giacomini:2017zju} and the universal $T$-matrix presentation of a quantum deformation of the (1+1) centrally extended Galilei Lie group. We have set here the mathematical foundations underlying the special relativistic generalization of this result.

We have found specifically that, under mild assumptions, there is a unique quantum deformation of the (1+1) centrally extended Poincaré Lie group whose non-relativistic counterpart is precisely the quantum Galilei group for QRF transformations of \cite{Ballesteros:2025ypr}. This limit is achieved by generalizing the Lie bialgebra approach to quantum algebra contractions developed in \cite{Ballesteros:1994iv, Ballesteros:1994at} to multiparametric Lie bialgebras and afterwards by introducing the contraction theory of universal $T$-matrices for the first time. In particular, starting from the Lie bialgebra associated with the most generic multiparametric Hopf algebra deformation of the (1+1) centrally extended Poincaré Lie algebra and assuming that the non-relativistic limit is given by the so-called fundamental multiparametric Lie bialgebra contraction, a unique quantum Poincaré algebra leading to the quantum Galilei algebra of interest is obtained. Afterwards, the Hopf algebra dual form associated with this quantum Poincaré deformation is derived and self-consistency of the method is proven, since the Galilei limit of such Poincaré $T$-matrix provides the universal $T$-matrix of \cite{Ballesteros:2025ypr} underlying non-relativistic QRF transformations. What is more, the Poincaré dual Hopf algebra is recognized, remarkably, as a non-trivial central extension of the (1+1) spacelike $\kappa$-Poincaré quantum group.

Additionally, from the mathematical perspective, it is worth emphasizing that a thorough analysis of Hopf algebra deformations of central extensions of Lie algebras and Lie groups has not been considered in the literature, despite the fact that they generate non-trivial features. It can be appreciated in \cite{Ballesteros:2018ghw, Ballesteros1999, Ballesteros:1999ew} that, even in the case where the Lie algebra central extension is trivial, it turns out that many quantum deformations of the extended Lie algebra cannot be written as a direct sum of Hopf algebras with the former central extension generating a Hopf subalgebra. As a consequence, the former central extension plays a highly non-trivial role at the quantum algebra and quantum group levels, generally converting into noncommutative. This is precisely the case of the quantum centrally extended Galilei and Poincaré algebras considered here, where both the former Lie algebra central extension and/or its corresponding Lie group local coordinate turn to noncommutative in a non-trivial way. In this regard, the associated Galilei and Minkowski embeddable noncommutative homogeneous extended spacetimes (see Eqs.~\eqref{EGalilei QG(a)} and \eqref{EPoincaré QG(a)}, respectively) have been shortly discussed within the approach of \cite{Ballesteros:2018ghw} where, in a broad sense, noncommutative spaces of this type were interpreted as noncommutative principal bundles over noncommutative spacetimes. In this context, the coordinate dual to the former Lie algebra central extension is interpreted as the coordinate on the fibers of a generally noncommutative principal bundle. It is also worth stressing that the classification of the (2+1) and (3+1) centrally extended Galilei and Poincaré Lie bialgebras and the corresponding quantum deformations is lacking in the literature. This classification would be essential in order to provide examples for the computation of the corresponding Galilei and Poincaré $T$-matrices in higher dimensions, which would be important, for instance, to investigate its potential role as QRF transformations beyond (1+1) dimensions.

Alternatively, from the QRF point of view, a challenge to be faced is the construction of a perspectival relativistic setting for inertial QRF transformations, which would be potentially realized by a quantum Poincaré group, or more specifically, by a Poincaré $T$-matrix. The aim of this paper is to present a consistent mathematical starting point to this end, motivated by the connection between non-relativistic QRF transformations and universal $T$-matrices provided in \cite{Ballesteros:2025ypr}. In this context, an appropriate relativistic counterpart of the framework of \cite{Giacomini:2017zju} should be introduced, and the phase space representations of the quantum Poincaré algebra and group entering the universal $T$-matrix should be studied. We recall that relativistic QRF transformations have also been considered within the different QRF approaches, including those introduced in \cite{Giacomini:2018gxh, delaHamette:2020dyi, Giacomini:2021gei, Apadula:2022pxk}, and some of them could deserve further investigation in light of the results of this work. It may be worth analyzing, for example, the physics of the noncommutative extended Minkowski spacetime following from the quantum Poincaré group introduced here, in terms of QRFs. Notice that, contrary to the Galilei case following from Eq.~\eqref{EGalilei QG(a)}, the Minkowski case following from Eq.~\eqref{EPoincaré QG(a)} contains a noncommutative time coordinate, which could induce, in a relativistic scenario, superpositions of proper times \cite{Giacomini:2021gei}.

Finally, the connection between quantum principal bundles and QRF transformations, recently reported in \cite{AliAhmad:2024vdw}, could give insight into the physics of the quantum coordinates related to the Galilei and Poincaré central extensions within this context.


\section*{Acknowledgements}

DFS acknowledges the hospitality of the Department Mathematik at Friedrich-Alexander-Universität Erlangen-Nürnberg and is grateful to Catherine Meusburger for helpful discussions related to this work. The authors acknowledge support from the Grant No. PID2023-148373NB-I00 funded by MCIN/AEI/10.13039/501100011033/FEDER, UE, by the Q-CAYLE Project funded by the Regional Government of Castilla y León (Junta de Castilla y León), and by the Ministry of Science and Innovation (MCIN) through the European Union funds NextGenerationEU (PRTR C17.I1). DFS acknowledges support from Universidad de Burgos through a PhD grant. This work contributes to the COST Action CA23130 ``Bridging high and low energies in search of quantum gravity (BridgeQG)''.



\newpage

\appendix

\section*{Appendix} \label{Appendix}

\addcontentsline{toc}{section}{Appendix}
\renewcommand{\theequation}{A.\arabic{equation}}
\setcounter{equation}{0}

In this Appendix, we consider the construction of the universal $T$-matrix for the Hopf algebra deformation of the universal enveloping algebra of the (1+1) centrally extended Poincaré Lie algebra of Proposition \ref{EPoincaré QA}, with coproducts and commutation relations reproduced here:
\begin{equation}
    \begin{aligned}
        \Delta(M) &= M \otimes 1 + e^{- \alpha P_1} \otimes M + (e^{- \alpha P_1} - 1) \otimes P_0 \ , \\
        \Delta(P_0) &= P_0 \otimes 1 + 1 \otimes P_0 \ , \\
        \Delta(P_1) &= P_1 \otimes 1 + 1 \otimes P_1 \ , \\
        \Delta(K) &= K \otimes e^{\alpha P_1} + 1 \otimes K \ ,
    \end{aligned}
\end{equation}
and
\begin{equation}
    \begin{gathered}
        \relax [P_0, P_1] = 0 \ , \quad [M, P_0] = 0 \ , \quad [M, P_1] = 0 \ , \\
        [K, P_0] = \frac{e^{\alpha P_1} - 1}{\alpha} \ , \quad [K, P_1] = (M + P_0) e^{\alpha P_1} \ , \\
        [K, M] = - \left(\frac{\cosh{(\alpha P_1)} - 1}{\alpha} + \frac{\alpha}{2} (M + P_0)^2 e^{\alpha P_1}\right) \ ,
    \end{gathered}
\end{equation}
respectively.

The basis elements are chosen in the following order:
\begin{equation}
    X^{a b c d} = M^a P_0^b P_1^c K^d \ ,
\end{equation}
with the dual basis elements defined by Eq.~\eqref{Duality}. In this case, $X^{1 0 0 0} = M$, $X^{0 1 0 0} = P_0$, $X^{0 0 1 0} = P_1$, $X^{0 0 0 1} = K$, so that the dual coordinate functions on the quantum group are defined locally as $x_{1 0 0 0} = \theta$, $x_{0 1 0 0} = a_0$, $x_{0 0 1 0} = a_1$, $x_{0 0 0 1} = \chi$. We have here that
\begin{equation}
    \begin{aligned}
        \Delta(X^{a b c d}) &= \Delta(M) \Delta(X^{(a - 1) b c d}) \ , \\
        \Delta(X^{a b c d}) &= \Delta(P_0) \Delta(X^{a (b - 1) c d}) \ ,  \\
        \Delta(X^{a b c d}) &= \Delta(P_1) \Delta(X^{a b (c - 1) d}) \ ,  \\
        \Delta(X^{a b c d}) &= \Delta(X^{a b c (d - 1)}) \Delta(K) \ , 
    \end{aligned}
\end{equation}
and the following recurrence relations for the structure tensor $F$ from Eq.~\eqref{FX}:
\begin{equation} \label{F1*}
    \begin{aligned}
        F_{i j k l ; p q r s}^{a b c d} &= F_{(i - 1) j k l ; p q r s}^{(a - 1) b c d} + \left(\sum_{n = 0}^k F_{i j n l ; (p - 1) q r s}^{(a - 1) b c d} + \sum_{n = 0}^{k - 1} F_{i j n l ; p (q - 1) r s}^{(a - 1) b c d}\right) \frac{(- \alpha)^{k - n}}{(k - n)!} \ , \\
        F_{i j k l ; p q r s}^{a b c d} &= F_{i (j - 1) k l ; p q r s}^{a (b - 1) c d} + F_{i j k l ; p (q - 1) r s}^{a (b - 1) c d} \ , \\
        F_{i j k l ; p q r s}^{a b c d} &= F_{i j (k - 1) l ; p q r s}^{a b (c - 1) d} + F_{i j k l ; p q (r - 1) s}^{a b (c - 1) d} \ .
    \end{aligned}
\end{equation}
The recurrence relation coming from $\Delta(X^{a b c d}) = \Delta(X^{a b c (d - 1)}) \Delta(K)$ is very difficult to find in general. We can show, specifically, that 
\begin{equation} \label{F2*}
    \begin{aligned}
        F_{1 0 0 0 ; p q r s}^{a b c d} &= a \, \delta_{p + 1}^a \, \delta_q^b \, \delta_r^c \, \delta_s^d \ , \\
        F_{0 1 0 0 ; p q r s}^{a b c d} &= b \, \delta_p^a \, \delta_{q + 1}^b \, \delta_r^c \, \delta_s^d \ , \\
        F_{0 0 1 0 ; p q r s}^{a b c d} &= c \, \delta_p^a \, \delta_q^b \, \delta_{r + 1}^c \, \delta_s^d \ , \\
        F_{i j k l ; 0 0 0 1}^{a b c d} &= d \, \delta_i^a \, \delta_j^b \, \delta_k^c \delta_{l + 1}^d \ .
    \end{aligned}
\end{equation}
We just need the above components of the structure tensor $F$ to compute the dual basis, so by taking into account Eq.~\eqref{Fx} together with Eq.~\eqref{F2*},
\begin{equation}
    x_{1 0 0 0} x_{(p - 1) q r s} = F_{1 0 0 0 ; (p - 1) q r s}^{a b c d} x_{a b c d} = p \, x_{p q r s} \ ,
\end{equation}
hence
\begin{equation}
    x_{p q r s} = \frac{\theta}{p} x_{(p - 1) q r s} = \dots = \frac{\theta^p}{p!} x_{0 q r s} \ .
\end{equation}
then, by repeating the same strategy,
\begin{equation}
	x_{a b c d} = \frac{\theta^a}{a!} \frac{a_0^b}{b!} \frac{a_1^c}{c!} \frac{\chi^d}{d!} \ ,
\end{equation}
therefore, the universal $T$-matrix \eqref{T0} is
\begin{equation}
    T = e^{\theta \otimes M} e^{a_0 \otimes P_0} e^{a_1 \otimes P_1} e^{\chi \otimes K} \ .
\end{equation}
This is the proof of Proposition \ref{EPoincaré T}.

The commutation relations for the coordinates on the dual quantum group follow from Eq.~\eqref{Fx} as well by
\begin{equation}
    [x_{i j k l}, x_{p q r s}] = (F_{i j k l ; p q r s}^{a b c d} - F_{p q r s ; i j k l}^{a b c d}) x_{a b c d} \ .
\end{equation}
The commutation relations $[\theta, a_0]$, $[\theta, a_1]$, $[a_0, a_1]$ are obtained from the recurrence relations \eqref{F1*}, but the remaining three, $[\chi, \theta]$, $[\chi, a_0]$, $[\chi, a_1]$, require the missing recurrence relation. We must look for the specific components of the structure tensor $F$ involved in them.

In the case of $[\chi, \theta]$, $F_{0 0 0 1 ; 1 0 0 0}^{a b c d}$ is, by duality, the coefficient of the term $K \otimes M$ in the coproduct of $X^{a b c d}$. This term appears when $a = 1$, $b = 0$, $c = 0$, $d = 1$, as well as when $a = 0$, $b = 0$, $c = 0$, and $d$ arbitrary. One can show that $F_{0 0 0 1 ; 1 0 0 0}^{1 0 0 1} = 1$ and $F_{0 0 0 1 ; 1 0 0 0}^{0 0 0 m} = \alpha$ for $m = 2 n$ (with $n = 1, 2, \dots$). On the other hand, $F_{1 0 0 0 ; 0 0 0 1}^{a b c d}$ is given by Eq.~\eqref{F2*}.

In the case of $[\chi, a_0]$, $F_{0 0 0 1 ; 0 1 0 0}^{a b c d}$ is, by duality, the coefficient of the term $K \otimes P_0$ in the coproduct of $X^{a b c d}$. This term appears when $a = 0$, $b = 1$, $c = 0$, $d = 1$, as well as when $a = 0$, $b = 0$, $c = 0$, and $d$ arbitrary. One can show that $F_{0 0 0 1 ; 0 1 0 0}^{0 1 0 1} = 1$ and $F_{0 0 0 1 ; 0 1 0 0}^{0 0 0 m} = \alpha$ for $m = 2 n$ (with $n = 1, 2, \dots$). On the other hand, $F_{0 1 0 0 ; 0 0 0 1}^{a b c d}$ is given by Eq.~\eqref{F2*}.

In the case of $[\chi, a_1]$, $F_{0 0 0 1 ; 0 0 1 0}^{a b c d}$ is, by duality, the coefficient of the term $K \otimes P_1$ in the coproduct of $X^{a b c d}$. This term appears when $a = 0$, $b = 0$, $c = 1$, $d = 1$, as well as when $a = 0$, $b = 0$, $c = 0$, and $d$ arbitrary. One can show that $F_{0 0 0 1 ; 0 0 1 0}^{0 0 1 1} = 1$ and $F_{0 0 0 1 ; 0 0 1 0}^{0 0 0 m} = \alpha$ for $m = 2 n + 1$ (with $n = 0, 1, 2, \dots$). On the other hand, $F_{0 0 1 0 ; 0 0 0 1}^{a b c d}$ is given by Eq.~\eqref{F2*}.

The coproducts could also be derived in a similar fashion, but given the universal $T$-matrix \eqref{EPoincaré T}, they follow from the corresponding group-like property of the universal $T$-matrix in the sense of Eq.~\eqref{T-duality}, simply a matrix multiplication when realizing the universal $T$-matrix as a true matrix. The fundamental representation of the quantum algebra of Proposition \ref{EPoincaré QA} is given by\footnote{This is also the fundamental representation of the (1+1) centrally extended Poincaré Lie algebra \eqref{EPoincaré LA}.}
\begin{equation}
    \scalebox{0.75}{$
    \rho(M) = \left(\begin{array}{cccc}
    0 & 0 & 0 & 0 \\
    0 & 0 & 0 & 0 \\
    0 & 0 & 0 & - 1 \\
    0 & 0 & 0 & 0
    \end{array}\right) \ , \quad
    \rho(P_0) = \left(\begin{array}{cccc}
    0 & 0 & 0 & - \frac{1}{2} \\
    0 & 0 & 0 & - \frac{1}{2} \\
    0 & 0 & 0 & 1 \\
    0 & 0 & 0 & 0
    \end{array}\right) \ , \quad 
    \rho(P_1) = \left(\begin{array}{cccc}
    0 & 0 & 0 & \frac{1}{2} \\
    0 & 0 & 0 & - \frac{1}{2} \\
    0 & 0 & 0 & 0 \\
    0 & 0 & 0 & 0
    \end{array}\right) \ , \quad 
    \rho(K) = \left(\begin{array}{cccc}
    - 1 & 0 & 0 & 0 \\
    0 & 1 & 0 & 0 \\
    0 & 0 & 0 & 0 \\
    0 & 0 & 0 & 0
    \end{array} \right) \ .
    $}
\end{equation}
As a consequence, the realization of the $T$-matrix \eqref{EPoincaré T} is
\begin{equation}
    T^\rho := (1 \otimes \rho) \, T = e^{\theta \otimes \rho(M)} e^{a_0 \otimes \rho(P_0)} e^{a_1 \otimes \rho(P_1)} e^{\chi \otimes \rho(K)} = \left(\begin{array}{cccc}
    e^{- \chi} & 0 & 0 & - \frac{1}{2}(a_0 - a_1) \\
    0 & e^\chi & 0 & - \frac{1}{2}(a_0 + a_1) \\
    0 & 0 & 1 & - (\theta - a_0) \\
    0 & 0 & 0 & 1
    \end{array}\right) \ ,
\end{equation}
which is just the centrally extended Poincaré Lie group element in a neighborhood of the identity with noncommutative entries. The coproduct map for its coordinate functions is read off from the matrix multiplication of two copies of $T^\rho$.

This is the sketch of the proof of Proposition \ref{EPoincaré QG}.


\newpage

\addcontentsline{toc}{section}{References}

\bibliographystyle{utphys2}

\bibliography{main}

@book{Chari1994,
    author = {Chari, V. and Pressley, A. N.},
    title = "{A Guide to Quantum Groups}",
    publisher = {Cambridge University Press},
    place = {Cambridge},
    year = {1994}
}

@book{Majid1995,
    author = {Majid, S.},
    title = "{Foundations of Quantum Group Theory}",
    publisher = {Cambridge University Press},
    place = {Cambridge},
    year = {1995}
}

@incollection{Fronsdal1994,
    author = "C. Fronsdal and A. Galindo",
    editor = "{P. J. Sally, M. Flato, J. Lepowsky, N. Reshetikhin, G. J. Zuckerman}",
    title = "{The universal $T$-matrix}",
    doi = "10.1090/conm/175",
    booktitle = "{Mathematical aspects of conformal and topological field theories and quantum groups}",
    publisher = "American Mathematical Society",
    pages = "73--88",
    year = "1994",
    isbn = "0-8218-5186-1",
}

@inproceedings{Celeghini1992,
    author = "Celeghini, E. and Giachetti, R. and Sorace, E. and Tarlini, M.",
    editor = "Kulish, Petr P.",
    title = "{Contractions of quantum groups}",
    doi = "10.1007/BFb0101192",
    booktitle = "Quantum Groups",
    publisher = "Springer Berlin Heidelberg",
    pages = "221--244",
    year = "1992",
    isbn = "978-3-540-47020-5"
}

@article{Poincare1900,
    author = "Poincaré, H.",
    title = "{Sur les groupes continus}",
    journal = "Trans. Cambr. Philos. Soc.",
    volume = "18",
    pages = "220--255",
    year = "1900"
}

@article{Birkhoff1937,
    author = "Birkhoff, G.",
    title = "{Representability of Lie algebras and Lie groups by matrices}",
    doi = "10.2307/1968569",
    journal = "Ann. Math.",
    volume = "38",
    pages = "526--532",
    year = "1937"
}

@article{Witt1937,
    author = "Witt, E.",
    title = "{Treue darstellung liescher ringe}",
    journal = "Jour. für die reine und angewandte Mathematik",
    doi = "10.1515/crll.1937.177.152",
    volume = "177",
    pages = "152--160",
    year = "1937"
}

@article{Groenewold:1946kp,
    author = "Groenewold, H. J.",
    title = "{On the principles of elementary quantum mechanics}",
    doi = "10.1016/S0031-8914(46)80059-4",
    journal = "Physica",
    volume = "12",
    pages = "405--460",
    year = "1946"
}

@article{Moyal:1949sk,
    author = "Moyal, J. E.",
    title = "{Quantum mechanics as a statistical theory}",
    doi = "10.1017/S0305004100000487",
    journal = "Proc. Cambridge Phil. Soc.",
    volume = "45",
    pages = "99--124",
    year = "1949"
}

@article{Inonu:1953sp,
    author = "Inonu, E. and Wigner, Eugene P.",
    title = "{On the contraction of groups and their representations}",
    doi = "10.1073/pnas.39.6.510",
    journal = "Proc. Nat. Acad. Sci.",
    volume = "39",
    pages = "510--524",
    year = "1953"
}

@article{Aharonov1967a,
    author = "Aharonov, Yakir and Susskind, Leonard",
    title = "{Charge superselection rule}",
    doi = "10.1103/PhysRev.155.1428",
    journal = "Phys. Rev.",
    volume = "155",
    pages = "1428--1431",
    year = "1967"
}

@article{DeWitt1967,
    author = "DeWitt, Bryce S",
    title = "{Quantum theory of gravity. I. The canonical theory}",
    doi = "10.1103/PhysRev.160.1113",
    journal = "Phys. Rev.",
    volume = "160",
    pages = "1113",
    year = "1967"
}

@article{Drinfeld:1983ky,
    author = "Drinfeld, V. G.",
    title = "{Hamiltonian structures of Lie groups, Lie bialgebras and the geometric meaning of the classical Yang-Baxter equations}",
    journal = "Sov. Math. Dokl.",
    volume = "27",
    pages = "68--71",
    year = "1983"
}

@article{Aharonov1984,
    author = "Aharonov, Y. and Kaufherr, T.",
    title = "{Quantum frames of reference}",
    doi = "10.1103/PhysRevD.30.368",
    journal = "Phys. Rev. D",
    volume = "30",
    pages = "368--385",
    year = "1984"
}

@article{Aldaya:1985plo,
    author = "Aldaya, V. and de Azcárraga, J. A.",
    title = "{Cohomology, central extensions, and (dynamical) groups}",
    doi = "10.1007/BF00672649",
    journal = "Int. J. Theor. Phys.",
    volume = "24",
    pages = "141--154",
    year = "1985"
}

@article{Drinfeld:1986in,
    author = "Drinfeld, V. G.",
    title = "{Quantum groups}",
    doi = "10.1007/BF01247086",
    journal = "Zap. Nauchn. Semin.",
    volume = "155",
    pages = "18--49",
    year = "1986"
}

@article{Tuynman:1987ij,
    author = "Tuynman, G. M. and Wiegerinck, W. A. J. J.",
    title = "{Central extensions and physics}",
    doi = "10.1016/0393-0440(87)90027-1",
    journal = "J. Geom. Phys.",
    volume = "4",
    pages = "207--258",
    year = "1987"
}

@article{Celeghini:1990bf,
    author = "Celeghini, E. and Giachetti, R. and Sorace, E. and Tarlini, M.",
    title = "{Three-dimensional quantum groups from contraction of $SU_q(2)$}",
    doi = "10.1063/1.529000",
    journal = "J. Math. Phys.",
    volume = "31",
    pages = "2548--2551",
    year = "1990"
}

@article{Celeghini:1990xx,
    author = "Celeghini, Enrico and Giachetti, Riccardo and Sorace, Emanuele and Tarlini, Marco",
    title = "{The three-dimensional Euclidean quantum group $E_q(3)$ and its R-matrix}",
    reportNumber = "PRINT-90-0691 (FLORENCE)",
    doi = "10.1063/1.529312",
    journal = "J. Math. Phys.",
    volume = "32",
    pages = "1159--1165",
    year = "1991"
}

@article{Rovelli:1990pi,
    author = "Rovelli, Carlo",
    title = "{Quantum reference systems}",
    reportNumber = "PITT-90-11",
    doi = "10.1088/0264-9381/8/2/012",
    journal = "Class. Quant. Grav.",
    volume = "8",
    pages = "317--332",
    year = "1991"
}

@article{Fronsdal:1991gf,
    author = "Fronsdal, C. and Galindo, A.",
    title = "{The dual of a quantum group}",
    reportNumber = "UCLA-91-TEP-54",
    doi = "10.1007/BF00739590",
    journal = "Lett. Math. Phys.",
    volume = "27",
    pages = "59--72",
    year = "1993"
}

@article{Lukierski:1991pn,
    author = "Lukierski, Jerzy and Ruegg, Henri and Nowicki, Anatol and Tolstoi, Valerii N.",
    title = "{Q deformation of Poincaré algebra}",
    reportNumber = "UGVA-DPT-1991-02-710",
    doi = "10.1016/0370-2693(91)90358-W",
    journal = "Phys. Lett. B",
    volume = "264",
    pages = "331--338",
    year = "1991"
}

@article{Lukierski:1991ff,
    author = "Lukierski, Jerzy and Nowicki, Anatol and Ruegg, Henri",
    title = "{Real forms of complex quantum anti-de Sitter algebra $U_q(Sp(4;C))$ and their contraction schemes}",
    eprint = "hep-th/9108018",
    archivePrefix = "arXiv",
    reportNumber = "UGVA-DPT-1991-08-740",
    doi = "10.1016/0370-2693(91)90094-7",
    journal = "Phys. Lett. B",
    volume = "271",
    pages = "321--328",
    year = "1991"
}

@article{Bonechi:1992cb,
    author = "Bonechi, F. and Celeghini, E. and Giachetti, R. and Sorace, E. and Tarlini, M.",
    title = "{Quantum Galilei group as symmetry of magnons}",
    eprint = "hep-th/9203048",
    archivePrefix = "arXiv",
    reportNumber = "DFF-156-03-92",
    doi = "10.1103/PhysRevB.46.5727",
    journal = "Phys. Rev. B",
    volume = "46",
    pages = "5727--5730",
    year = "1992"
}

@article{Bonechi:1992ye,
    author = "Bonechi, F. and Celeghini, E. and Giachetti, R. and Sorace, E. and Tarlini, M.",
    title = "{Heisenberg XXZ model and quantum Galilei group}",
    eprint = "hep-th/9204054",
    archivePrefix = "arXiv",
    reportNumber = "DFF-160-04-92",
    doi = "10.1088/0305-4470/25/15/007",
    journal = "J. Phys. A",
    volume = "25",
    pages = "L939--L943",
    year = "1992"
}

@article{Brzezinski:1992hda,
    author = "Brzezinski, Tomasz and Majid, Shahn",
    title = "{Quantum group gauge theory on quantum spaces}",
    eprint = "hep-th/9208007",
    archivePrefix = "arXiv",
    reportNumber = "DAMTP-92-27",
    doi = "10.1007/BF02096884",
    journal = "Commun. Math. Phys.",
    volume = "157",
    pages = "591--638",
    year = "1993",
    note = "[Erratum: Commun.Math.Phys. 167, 235 (1995)]"
}

@article{Brzezinski:1992ut,
    author = "Brzezinski, Tomasz and Majid, Shahn",
    title = "{Quantum group gauge theory on classical spaces}",
    eprint = "hep-th/9210024",
    archivePrefix = "arXiv",
    reportNumber = "DAMTP-92-51",
    doi = "10.1016/0370-2693(93)91830-G",
    journal = "Phys. Lett. B",
    volume = "298",
    pages = "339--343",
    year = "1993"
}

@article{Giller:1992xg,
    author = "Giller, S. and Kosinski, P. and Majewski, M. and Maslanka, P. and Kunz, J.",
    title = "{More about Q deformed Poincaré algebra}",
    reportNumber = "KFT-UL-2-92",
    doi = "10.1016/0370-2693(92)90158-Z",
    journal = "Phys. Lett. B",
    volume = "286",
    pages = "57--62",
    year = "1992"
}

@article{Lukierski:1992dt,
    author = "Lukierski, Jerzy and Nowicki, Anatol and Ruegg, Henri",
    title = "{New quantum Poincaré algebra and k deformed field theory}",
    reportNumber = "UGVA-DPT-1992-07-776, LPTB-92-04",
    doi = "10.1016/0370-2693(92)90894-A",
    journal = "Phys. Lett. B",
    volume = "293",
    pages = "344--352",
    year = "1992"
}

@article{Semenov-Tian-Shansky:1992xxt,
    author = "Semenov-Tian-Shansky, M. A.",
    title = "{Poisson-Lie groups, quantum duality principle, and the quantum double}",
    eprint = "hep-th/9304042",
    archivePrefix = "arXiv",
    reportNumber = "PRINT-93-0378 (STEKLOV)",
    doi = "10.1007/BF01083527",
    journal = "Theor. Math. Phys.",
    volume = "93",
    pages = "1292--1307",
    year = "1992"
}

@article{Bonechi:1993sn,
    author = "Bonechi, F. and Celeghini, E. and Giachetti, R. and Perena, C. M. and Sorace, E. and Tarlini, M.",
    title = "{Exponential mapping for nonsemisimple quantum groups}",
    eprint = "hep-th/9311114",
    archivePrefix = "arXiv",
    reportNumber = "DFF-192-9-93",
    doi = "10.1088/0305-4470/27/4/023",
    journal = "J. Phys. A",
    volume = "27",
    pages = "1307--1316",
    year = "1994"
}

@article{Budzynski:1993dm,
    author = "Budzynski, Robert J. and Kondracki, Witold",
    title = "{Quantum principal fiber bundles: Topological aspects}",
    eprint = "hep-th/9401019",
    archivePrefix = "arXiv",
    reportNumber = "IM-PAN-517",
    year = "1993"
}

@article{Cangemi:1993sd,
    author = "Cangemi, D. and Jackiw, R.",
    title = "{Poincaré gauge theory for gravitational forces in (1+1) dimensions}",
    eprint = "hep-th/9302026",
    archivePrefix = "arXiv",
    reportNumber = "MIT-CTP-2165",
    doi = "10.1006/aphy.1993.1058",
    journal = "Annals Phys.",
    volume = "225",
    pages = "229--263",
    year = "1993"
}

@article{Nappi:1993ie,
    author = "Nappi, Chiara R. and Witten, Edward",
    title = "{A WZW model based on a nonsemisimple group}",
    eprint = "hep-th/9310112",
    archivePrefix = "arXiv",
    reportNumber = "IASSNS-HEP-93-61",
    doi = "10.1103/PhysRevLett.71.3751",
    journal = "Phys. Rev. Lett.",
    volume = "71",
    pages = "3751--3753",
    year = "1993"
}

@article{Pflaum:1993mt,
    author = "Pflaum, Markus J.",
    title = "{Quantum groups on fiber bundles}",
    eprint = "hep-th/9401085",
    archivePrefix = "arXiv",
    doi = "10.1007/BF02112317",
    journal = "Commun. Math. Phys.",
    volume = "166",
    pages = "279--316",
    year = "1994"
}

@article{Ballesteros:1994at,
    author = "Ballesteros, A. and Gromov, Nikolai A. and Herranz, F. J. and del Olmo, M. A. and Santander, M.",
    title = "{Lie bialgebra contractions and quantum deformations of quasiorthogonal algebras}",
    eprint = "hep-th/9412083",
    archivePrefix = "arXiv",
    reportNumber = "UVA-011194",
    doi = "10.1063/1.531368",
    journal = "J. Math. Phys.",
    volume = "36",
    pages = "5916--5937",
    year = "1995"
}

@article{Ballesteros:1994iv,
    author = "Ballesteros, A. and Herranz, F. J. and del Olmo, M. A. and Santander, M.",
    title = "{Classical deformations, Poisson-Lie contractions, and quantization of dual Lie bialgebras}",
    eprint = "hep-th/9403182",
    archivePrefix = "arXiv",
    reportNumber = "UVA-94-302",
    doi = "10.1063/1.531331",
    journal = "J. Math. Phys.",
    volume = "36",
    pages = "631--640 ",
    year = "1995"
}

@article{Chakrabarti:1994hp,
    author = "Chakrabarti, R. and Jagannathan, R.",
    title = "{On the Hopf structure of $U_{p,q}(gl(1|1))$ and the universal T-matrix of $Fun_p,q(GL(1|1))$}",
    eprint = "hep-th/9409161",
    archivePrefix = "arXiv",
    reportNumber = "IC-94-254",
    doi = "10.1007/BF00416022",
    journal = "Lett. Math. Phys.",
    volume = "37",
    pages = "191--199",
    year = "1996"
}

@article{Damaskinsky:1994au,
    author = "Damaskinsky, E. V. and Sokolov, M. A.",
    title = "{Some remarks on the Gauss decomposition for quantum group $GL_q(n)$ with application to q-bosonization}",
    eprint = "hep-th/9407024",
    archivePrefix = "arXiv",
    doi = "10.1088/0305-4470/28/13/017",
    journal = "J. Phys. A",
    volume = "28",
    pages = "3725--3732",
    year = "1995"
}

@article{Doplicher:1994zv,
    author = "Doplicher, S. and Fredenhagen, K. and Roberts, J. E.",
    title = "{Space-time quantization induced by classical gravity}",
    reportNumber = "DESY-94-065",
    doi = "10.1016/0370-2693(94)90940-7",
    journal = "Phys. Lett. B",
    volume = "331",
    pages = "39--44",
    year = "1994"
}

@article{Doplicher:1994tu,
    author = "Doplicher, Sergio and Fredenhagen, Klaus and Roberts, John E.",
    title = "{The quantum structure of space-time at the Planck scale and quantum fields}",
    eprint = "hep-th/0303037",
    archivePrefix = "arXiv",
    doi = "10.1007/BF02104515",
    journal = "Commun. Math. Phys.",
    volume = "172",
    pages = "187--220",
    year = "1995"
}

@article{Jagannathan:1994cm,
    author = "Jagannathan, R. and Van der Jeugt, J.",
    title = "{Finite dimensional representations of the quantum group $GL_{p,q}(2)$ using the exponential map from $U_{p,q}(gl(2))$}",
    eprint = "hep-th/9411200",
    archivePrefix = "arXiv",
    doi = "10.1088/0305-4470/28/10/013",
    journal = "J. Phys. A",
    volume = "28",
    pages = "2819--2832",
    year = "1995"
}

@article{Jurco:1994cx,
    author = "Jurco, B. and Stovicek, P.",
    title = "{Coherent states for quantum compact groups}",
    eprint = "hep-th/9403114",
    archivePrefix = "arXiv",
    reportNumber = "CERN-TH-7201-94",
    doi = "10.1007/BF02506391",
    journal = "Commun. Math. Phys.",
    volume = "182",
    pages = "221--251",
    year = "1996"
}

@article{Morozov:1994ab,
    author = "Morozov, Alexi and Vinet, Luc",
    title = "{Free field representation of group element for simple quantum groups}",
    eprint = "hep-th/9409093",
    archivePrefix = "arXiv",
    reportNumber = "ITEP-M3-94, CRM-2202",
    doi = "10.1142/S0217751X9800072X",
    journal = "Int. J. Mod. Phys. A",
    volume = "13",
    pages = "1651--1708",
    year = "1998"
}

@article{Zakrzewski:1994hlc,
    author = "Zakrzewski, S.",
    title = "{Quantum Poincaré group related to the $\kappa$-Poincaré algebra}",
    doi = "10.1088/0305-4470/27/6/030",
    journal = "J. Phys. A",
    volume = "27",
    pages = "2075",
    year = "1994"
}

@article{Ballesteros1995,
    author = "Ballesteros, A. and Herranz, F. J. and del Olmo, M. A. and Perena, C. M. and Santander, M.",
    title = "{Non-standard quantum (1+1) Poincaré group: a T-matrix approach}",
    eprint = "q-alg/9501029",
    archivePrefix = "arXiv",
    doi = "10.1088/0305-4470/28/24/012",
    journal = "J. Phys. A: Math. Gen.",
    volume = "28",
    pages = "7113--7125",
    year = "1995"
}

@article{Brown:1995fj,
    author = "Brown, J. David and Marolf, Donald",
    title = "{On relativistic material reference systems}",
    eprint = "gr-qc/9509026",
    archivePrefix = "arXiv",
    primaryClass = "gr-qc",
    reportNumber = "UCSBTH-95-26, CTMP-012-NCSU",
    doi = "10.1103/PhysRevD.53.1835",
    journal = "Phys. Rev. D",
    volume = "53",
    pages = "1835--1844",
    year = "1996"
}

@article{deAzcarraga:1995uw,
    author = "de Azcarraga, J. A. and Perez Bueno, J. C.",
    title = "{Relativistic and Newtonian $\kappa$-spacetimes}",
    eprint = "q-alg/9505004",
    archivePrefix = "arXiv",
    reportNumber = "FTUV-95-12, IFIC-95-12, FTUV-95-12---IFIC-95-12",
    doi = "10.1063/1.531196",
    journal = "J. Math. Phys.",
    volume = "36",
    pages = "6879--6896",
    year = "1995"
}

@article{Fronsdal1995a,
    author = "Fronsdal, C.",
    title = "{Universal T-matrix for twisted quantum $gl(n)$}",
    eprint = "q-alg/9505014",
    archivePrefix = "arXiv",
    year = "1995"
}

@article{Fronsdal1995b,
    author = "Fronsdal, C. and Galindo, A.",
    title = "{Deformations of multiparameter quantum $GL(n)$}",
    eprint = "q-alg/9505016",
    archivePrefix = "arXiv",
    doi = "10.1007/BF00739372",
    journal = "Lett. Math. Phys.",
    volume = "34",
    pages = "25--36",
    year = "1995"
}

@article{VanDerJeugt:1995yn,
    author = "Van Der Jeugt, J. and Jagannathan, R.",
    editor = "Burdik, C. and Lukierski, J. and Gazeau, J. P. and Havlicek, M. and Tolar, J.",
    title = "{The exponential map for representations of $U_{p,q}(gl(2))$}",
    eprint = "q-alg/9507009",
    archivePrefix = "arXiv",
    doi = "10.1007/BF01688821",
    journal = "Czech. J. Phys.",
    volume = "46",
    pages = "269",
    year = "1996"
}

@article{Ballesteros:1996awf,
    author = "Ballesteros, A. and Herranz, F. J. and Pere\~na, C. M.",
    title = "{Null-plane quantum universal R-matrix}",
    eprint = "q-alg/9607009",
    archivePrefix = "arXiv",
    reportNumber = "UBU-DFIS-96-06",
    doi = "10.1016/S0370-2693(96)01435-9",
    journal = "Phys. Lett. B",
    volume = "391",
    pages = "71--77",
    year = "1997"
}

@article{Chakrabarti1996,
    author = "Chakrabarti, R. and Jagannathan, R.",
    title = "{The dual $(p,q)$-Alexander-Conway Hopf algebras and the associated universal T-matrix}",
    eprint = "q-alg/9602020",
    archivePrefix = "arXiv",
    doi = "10.1007/s002880050273",
    journal = "Z. Phys. C",
    volume = "72",
    pages = "519–524",
    year = "1996"
}

@article{Feinsilver1997,
    author = "Feinsilver, P. and Franz, U. and Schott, R.",
    title = "{Duality and multiplicative stochastic processes on quantum groups}",
    doi = "10.1023/A:1022618114810",
    journal = "J. Theor. Probab.",
    volume = "10",
    pages = "795--818",
    year = "1997"
}

@article{Opanowicz:1998zz,
    author = "Opanowicz, Anna",
    title = "{Lie bialgebra structures for centrally extended two-dimensional Galilei algebra and their Lie-Poisson counterparts}",
    eprint = "q-alg/9710028",
    archivePrefix = "arXiv",
    doi = "10.1088/0305-4470/31/41/012",
    journal = "J. Phys. A",
    volume = "31",
    pages = "8387--8396",
    year = "1998"
}

@article{Quesne1998,
    author = "Quesne, C.",
    title = "{Duals of colored quantum universal enveloping algebras and colored universal T-matrices}",
    eprint = "q-alg/9711012",
    archivePrefix = "arXiv",
    doi = "10.1063/1.532378",
    journal = "J. Math. Phys.",
    volume = "39",
    pages = "1199–1222",
    year = "1998"
}

@article{Ahmedov1999,
    author = "Ahmedov, H. and Dayi, O. F.",
    title = "{$SL_q(2,R)$ at roots of unity}",
    eprint = "math/9809052",
    archivePrefix = "arXiv",
    doi = "10.1088/0305-4470/32/10/008",
    journal = "J. Phys. A: Math. Gen.",
    volume = "32",
    pages = "1895--1907",
    year = "1999"
}

@article{Ballesteros1999,
    author = "Ballesteros, A. and Herranz, F. J. and Parashar, P.",
    title = "{Multiparametric quantum $gl(2)$: Lie bialgebras, quantum R-matrices and non-relativistic limits}",
    eprint = "math/9806149",
    archivePrefix = "arXiv",
    doi = "10.1088/0305-4470/32/12/010",
    journal = "J. Phys. A: Math. Gen.",
    volume = "32",
    pages = "2369--2385",
    year = "1999"
}

@article{Ballesteros:1999ew,
    author = "Ballesteros, Angel and Celeghini, Enrico and Herranz, Francisco J.",
    title = "{Quantum (1+1) extended Galilei algebras: from Lie bialgebras to quantum R-matrices and integrable systems}",
    eprint = "math/9906094",
    archivePrefix = "arXiv",
    reportNumber = "UBU-DFIS-99-04",
    doi = "10.1088/0305-4470/33/17/303",
    journal = "J. Phys. A",
    volume = "33",
    pages = "3431--3444",
    year = "2000"
}

@article{Chakrabarti1999,
    author = "Chakrabarti, R. and Quesne, C.",
    title = "{On Jordanian $U_{h,\alpha}(gl(2))$ algebra and its T-matrices via a contraction method}",
    eprint = "math/9811064",
    archivePrefix = "arXiv",
    doi = "10.1142/S0217751X9900124X",
    journal = "Int. J. Mod. Phys. A",
    volume = "14",
    pages = "2511--2529",
    year = "1999"
}

@article{Opanowicz:1999qp,
    author = "Opanowicz, Anna",
    title = "{Two-dimensional centrally extended quantum Galilei groups and their algebras}",
    eprint = "math/9905141",
    archivePrefix = "arXiv",
    doi = "10.1088/0305-4470/33/9/316",
    journal = "J. Phys. A",
    volume = "33",
    pages = "1941--1960",
    year = "2000"
}

@article{Aizawa2000,
    author = "Aizawa, N.",
    title = "{Representation functions for Jordanian quantum group $SL_h(2)$ and Jacobi polynomials}",
    doi = "10.1088/0305-4470/33/20/302",
    eprint = "math/0003110",
    archivePrefix = "arXiv",
    journal = "J. Phys. A: Math. Gen.",
    volume = "33",
    pages = "3735--3752",
    year = "2000"
}

@article{Sokolov2000,
    author = "Sokolov, M. A.",
    title = "{Exponential representation of Jordanian matrix quantum group $GL_h(2)$}",
    doi = "10.1088/0305-4470/33/2/314",
    journal = "J. Phys. A: Math. Gen.",
    volume = "33",
    pages = "417--425",
    year = "2000"
}

@article{Douglas:2001ba,
    author = "Douglas, Michael R. and Nekrasov, Nikita A.",
    title = "{Noncommutative field theory}",
    eprint = "hep-th/0106048",
    archivePrefix = "arXiv",
    reportNumber = "ITEP-TH-31-01, IHES-P-01-27, RUNHETC-2001-18",
    doi = "10.1103/RevModPhys.73.977",
    journal = "Rev. Mod. Phys.",
    volume = "73",
    pages = "977--1029",
    year = "2001"
}

@article{Szabo:2001kg,
    author = "Szabo, Richard J.",
    title = "{Quantum field theory on noncommutative spaces}",
    eprint = "hep-th/0109162",
    archivePrefix = "arXiv",
    reportNumber = "HWM-01-35, EMPG-01-14",
    doi = "10.1016/S0370-1573(03)00059-0",
    journal = "Phys. Rept.",
    volume = "378",
    pages = "207--299",
    year = "2003"
}

@article{Neeb2002,
    author = "Neeb, K. H.",
    title = "{Central extensions of infinite-dimensional Lie groups}",
    doi = "10.5802/aif.1921",
    journal = "Ann. Inst. Fourier",
    volume = "52",
    pages = "1365--1442",
    year = "2002"
}

@article{Aizawa:2005fq,
    author = "Aizawa, N. and Chakrabaarti, R. and Segar, J.",
    title = "{Generalized boson algebra and its entangled bipartite coherent states}",
    eprint = "quant-ph/0509031",
    archivePrefix = "arXiv",
    doi = "10.1088/0305-4470/38/41/012",
    journal = "J. Phys. A",
    volume = "38",
    pages = "9007--9018",
    year = "2005"
}

@article{Aizawa2006,
    author = "Aizawa, N. and Chakrabarti, R. and Mohammed, S. S. Naina and Segar, J.",
    title = "{Universal T-matrix, representations of $OSp_q(1/2)$ and little q-Jacobi polynomials}",
    eprint = "math/0607566",
    archivePrefix = "arXiv",
    doi = "10.1063/1.2399360",
    journal = "J. Math. Phys.",
    volume = "47",
    pages = "123511",
    year = "2006"
}

@article{Poulin:2006ryq,
    author = "Poulin, David and Yard, Jon",
    title = "{Dynamics of a quantum reference frame}",
    eprint = "quant-ph/0612126",
    archivePrefix = "arXiv",
    doi = "10.1088/1367-2630/9/5/156",
    journal = "New J. Phys.",
    volume = "9",
    pages = "156--156",
    year = "2007"
}

@article{Bartlett2007,
    author = "Bartlett, Stephen D. and Rudolph, Terry and Spekkens, Robert W.",
    title = "{Reference frames, superselection rules, and quantum information}",
    eprint = "0610030",
    archivePrefix = "arXiv",
    primaryClass = "quant-ph",
    doi = "10.1103/RevModPhys.79.555",
    journal = "Rev. Mod. Phys.",
    volume = "79",
    pages = "555--609",
    year = "2007"
}

@article{Girelli:2007xn,
    author = "Girelli, Florian and Poulin, David",
    title = "{Quantum reference frames and deformed symmetries}",
    eprint = "0710.4393",
    archivePrefix = "arXiv",
    primaryClass = "gr-qc",
    doi = "10.1103/PhysRevD.77.104012",
    journal = "Phys. Rev. D",
    volume = "77",
    pages = "104012",
    year = "2008"
}

@article{Aizawa2009,
    author = "Aizawa, N. and Chakrabarti, R.",
    title = "{Coherent state on $SU_q(2)$ homogeneous space}",
    eprint = "0905.0194",
    archivePrefix = "arXiv",
    primaryClass = "math.QA",
    doi = "10.1088/1751-8113/42/29/295208",
    journal = "J. Phys. A",
    volume = "42",
    pages = "295208",
    year = "2009"
}

@article{Angelo2011,
    author = "Renato M Angelo and Nicolas Brunner and Sandu Popescu and Anthony J Short and Paul Skrzypczyk",
    title = "{Physics within a quantum reference frame}",
    eprint = "1007.2292",
    archivePrefix = "arXiv",
    primaryClass = "quant-ph",
    doi = "10.1088/1751-8113/44/14/145304",
    journal = "J. Phys. A",
    volume = "44",
    pages = "145304",
    year = "2011"
}

@article{Tambornino:2011vg,
    author = "Tambornino, Johannes",
    title = "{Relational observables in gravity: A review}",
    eprint = "1109.0740",
    archivePrefix = "arXiv",
    primaryClass = "gr-qc",
    doi = "10.3842/SIGMA.2012.017",
    journal = "SIGMA",
    volume = "8",
    pages = "017",
    year = "2012"
}

@article{Ballesteros2013,
    author = "Ballesteros, A. and Musso, F.",
    title = "{Quantum algebras as quantizations of dual Poisson-Lie groups}",
    eprint = "1212.3809",
    archivePrefix = "arXiv",
    primaryClass = "math-ph",
    doi = "10.1088/1751-8113/46/19/195203",
    journal = "J. Phys. A",
    volume = "46",
    pages = "195203",
    year = "2013"
}

@article{Palmer:2013zza,
    author = "Palmer, Matthew C. and Girelli, Florian and Bartlett, Stephen D.",
    title = "{Changing quantum reference frames}",
    eprint = "1307.6597",
    archivePrefix = "arXiv",
    primaryClass = "quant-ph",
    doi = "10.1103/PhysRevA.89.052121",
    journal = "Phys. Rev. A",
    volume = "89",
    pages = "052121",
    year = "2014"
}

@article{Busch2016,
    author = "Takayuki Miyadera and Leon Loveridge and Paul Busch",
    title = "{Approximating relational observables by absolute quantities: A quantum accuracy-size trade-off}",
    eprint = "1510.02063",
    archivePrefix = "arXiv",
    primaryClass = "quant-ph",
    doi = "10.1088/1751-8113/49/18/185301",
    journal = "J. Phys. A",
    volume = "49",
    pages = "185301",
    year = "2016"
}

@article{Smith2016,
    author = "Smith, Alexander R. H. and Piani, Marco and Mann, Robert B.",
    title = "{Quantum reference frames associated with noncompact groups: The case of translations and boosts and the role of mass}",
    eprint = "1602.07696",
    archivePrefix = "arXiv",
    primaryClass = "quant-ph",
    doi = "10.1103/PhysRevA.94.012333",
    journal = "Phys. Rev. A",
    volume = "94",
    pages = "012333",
    year = "2016"
}

@article{Giacomini:2017zju,
    author = "Giacomini, Flaminia and Castro-Ruiz, Esteban and Brukner, Caslav",
    title = "{Quantum mechanics and the covariance of physical laws in quantum reference frames}",
    eprint = "1712.07207",
    archivePrefix = "arXiv",
    primaryClass = "quant-ph",
    doi = "10.1038/s41467-019-08754-0",
    journal = "Nat. Commun.",
    volume = "10",
    pages = "494",
    year = "2019"
}

@article{Ballesteros:2018ghw,
    author = "Ballesteros, Angel and Mercati, Flavio",
    title = "{Extended noncommutative Minkowski spacetimes and hybrid gauge symmetries}",
    eprint = "1805.07099",
    archivePrefix = "arXiv",
    primaryClass = "hep-th",
    doi = "10.1140/epjc/s10052-018-6097-1",
    journal = "Eur. Phys. J. C",
    volume = "78",
    pages = "615",
    year = "2018"
}

@article{Giacomini:2018gxh,
    author = "Giacomini, Flaminia and Castro-Ruiz, Esteban and Brukner, Caslav",
    title = "{Relativistic quantum reference frames: The operational meaning of spin}",
    eprint = "1811.08228",
    archivePrefix = "arXiv",
    primaryClass = "quant-ph",
    doi = "10.1103/PhysRevLett.123.090404",
    journal = "Phys. Rev. Lett.",
    volume = "123",
    pages = "090404",
    year = "2019"
}

@article{Lizzi:2018qaf,
    author = "Lizzi, Fedele and Manfredonia, Mattia and Mercati, Flavio and Poulain, Timothé",
    title = "{Localization and reference frames in $\kappa$-Minkowski spacetime}",
    eprint = "1811.08409",
    archivePrefix = "arXiv",
    primaryClass = "hep-th",
    doi = "10.1103/PhysRevD.99.085003",
    journal = "Phys. Rev. D",
    volume = "99",
    pages = "085003",
    year = "2019"
}

@article{Lizzi:2019wto,
    author = "Lizzi, Fedele and Manfredonia, Mattia and Mercati, Flavio",
    title = "{Localizability in $\kappa$-Minkowski spacetime}",
    eprint = "1912.07098",
    archivePrefix = "arXiv",
    primaryClass = "hep-th",
    doi = "10.1142/S0219887820400101",
    journal = "Int. J. Geom. Meth. Mod. Phys.",
    volume = "17",
    pages = "2040010",
    year = "2020"
}

@article{Ballesteros:2020lgl,
    author = "Ballesteros, Angel and Giacomini, Flaminia and Gubitosi, Giulia",
    title = "{The group structure of dynamical transformations between quantum reference frames}",
    eprint = "2012.15769",
    archivePrefix = "arXiv",
    primaryClass = "quant-ph",
    doi = "10.22331/q-2021-06-08-470",
    journal = "Quantum",
    volume = "5",
    pages = "470",
    year = "2021"
}

@article{delaHamette:2020dyi,
    author = "de la Hamette, Anne-Catherine and Galley, Thomas D.",
    title = "{Quantum reference frames for general symmetry groups}",
    eprint = "2004.14292",
    archivePrefix = "arXiv",
    doi = "10.22331/q-2020-11-30-367",
    journal = "Quantum",
    volume = "4",
    pages = "367",
    year = "2020"
}

@article{Hoehn:2020epv,
    author = "Hoehn, Philipp A. and Smith, Alexander R. H. and Lock, Maximilian P. E.",
    title = "{Equivalence of approaches to relational quantum dynamics in relativistic settings}",
    eprint = "2007.00580",
    archivePrefix = "arXiv",
    primaryClass = "gr-qc",
    doi = "10.3389/fphy.2021.587083",
    journal = "Front. in Phys.",
    volume = "9",
    pages = "181",
    year = "2021"
}

@article{Krumm:2020fws,
    author = "Krumm, Marius and Hoehn, Philipp A. and Mueller, Markus P.",
    title = "{Quantum reference frame transformations as symmetries and the paradox of the third particle}",
    eprint = "2011.01951",
    archivePrefix = "arXiv",
    primaryClass = "quant-ph",
    doi = "10.22331/q-2021-08-27-530",
    journal = "Quantum",
    volume = "5",
    pages = "530",
    year = "2021"
}

@article{Giacomini:2021gei,
    author = "Giacomini, Flaminia",
    title = "{Spacetime quantum reference frames and superpositions of proper times}",
    eprint = "2101.11628",
    archivePrefix = "arXiv",
    primaryClass = "quant-ph",
    doi = "10.22331/q-2021-07-22-508",
    journal = "Quantum",
    volume = "5",
    pages = "508",
    year = "2021"
}

@article{Mikusch:2021kro,
    author = "Mikusch, Marion and Barbado, Luis C. and Brukner, Caslav",
    title = "{Transformation of spin in quantum reference frames}",
    eprint = "2103.05022",
    archivePrefix = "arXiv",
    primaryClass = "quant-ph",
    doi = "10.1103/PhysRevResearch.3.043138",
    journal = "Phys. Rev. Res.",
    volume = "3",
    pages = "043138",
    year = "2021"
}

@article{Streiter:2021kta,
    author = "Streiter, Lucas F. and Giacomini, Flaminia and Brukner, Caslav",
    title = "{Relativistic Bell test within quantum reference frames}",
    eprint = "2008.03317",
    archivePrefix = "arXiv",
    primaryClass = "quant-ph",
    doi = "10.1103/PhysRevLett.126.230403",
    journal = "Phys. Rev. Lett.",
    volume = "126",
    pages = "230403",
    year = "2021"
}

@article{Apadula:2022pxk,
    author = "Apadula, Luca and Castro-Ruiz, Esteban and Brukner, Caslav",
    title = "{Quantum reference frames for Lorentz symmetry}",
    eprint = "2212.14081",
    archivePrefix = "arXiv",
    primaryClass = "quant-ph",
    doi = "10.22331/q-2024-08-14-1440",
    journal = "Quantum",
    volume = "8",
    pages = "1440",
    year = "2024"
}

@article{Girelli:2022foc,
    author = "Girelli, Florian and Laudonio, Matteo",
    title = "{Group field theory on quantum groups}",
    eprint = "2205.13312",
    archivePrefix = "arXiv",
    primaryClass = "hep-th",
    year = "2022"
}

@article{Lizzi:2022hcq,
    author = "Lizzi, Fedele and Scala, Luca and Vitale, Patrizia",
    title = "{Localization and observers in $\varrho$-Minkowski spacetime}",
    eprint = "2205.10862",
    archivePrefix = "arXiv",
    primaryClass = "hep-th",
    doi = "10.1103/PhysRevD.106.025023",
    journal = "Phys. Rev. D",
    volume = "106",
    pages = "025023",
    year = "2022"
}

@article{Mertens:2022aou,
    author = "Mertens, Thomas G.",
    title = "{Quantum exponentials for the modular double and applications in gravity models}",
    eprint = "2212.07696",
    archivePrefix = "arXiv",
    primaryClass = "hep-th",
    doi = "10.1007/JHEP09(2023)106",
    journal = "JHEP",
    volume = "09",
    pages = "106",
    year = "2023"
}

@article{Glowacki:2023nnf,
    author = "Glowacki, Jan and Loveridge, Leon and Waldron, James",
    title = "{Quantum reference frames on finite homogeneous spaces}",
    eprint = "2302.05354",
    archivePrefix = "arXiv",
    primaryClass = "quant-ph",
    doi = "10.1007/s10773-024-05650-7",
    journal = "Int. J. Theor. Phys.",
    volume = "63",
    pages = "137",
    year = "2024"
}

@article{Lake:2023nua,
    author = "Lake, Matthew J. and Miller, Marek",
    title = "{Quantum reference frames, revisited}",
    eprint = "2312.03811",
    archivePrefix = "arXiv",
    primaryClass = "gr-qc",
    year = "2023"
}

@article{AliAhmad:2024vdw,
    author = "Ali Ahmad, Shadi and Chemissany, Wissam and Klinger, Marc S. and Leigh, Robert G.",
    title = "{Relational quantum geometry}",
    eprint = "2410.11029",
    archivePrefix = "arXiv",
    primaryClass = "hep-th",
    doi = "10.1016/j.nuclphysb.2025.116911",
    journal = "Nucl. Phys. B",
    volume = "1015",
    pages = "116911",
    year = "2025"
}

@article{DEsposito:2024wru,
    author = "D'Esposito, Vittorio and Fabiano, Giuseppe and Frattulillo, Domenico and Mercati, Flavio",
    title = "{Doubly quantum mechanics}",
    eprint = "2412.05997",
    archivePrefix = "arXiv",
    primaryClass = "quant-ph",
    doi = "10.22331/q-2025-04-24-1721",
    journal = "Quantum",
    volume = "9",
    pages = "1721",
    year = "2025"
}

@article{Fewster:2024pur,
    author = "Fewster, J. Christopher and Janssen, Daan W. and Loveridge, Leon Deryck and Rejzner, Kasia and Waldron, James",
    title = "{Quantum reference frames, measurement schemes and the type of local algebras in quantum field theory}",
    eprint = "2403.11973",
    archivePrefix = "arXiv",
    primaryClass = "math-ph",
    doi = "10.1007/s00220-024-05180-7",
    journal = "Commun. Math. Phys.",
    volume = "406",
    pages = "19",
    year = "2025"
}

@article{Ballesteros:2025ypr,
    author = "Ballesteros, Angel and Fernandez-Silvestre, Diego and Giacomini, Flaminia and Gubitosi, Giulia",
    title = "{Quantum Galilei group as quantum reference frame transformations}",
    eprint = "2504.00569",
    archivePrefix = "arXiv",
    primaryClass = "quant-ph",
    doi = "10.22331/q-2025-12-10-1935",
    journal = "Quantum",
    volume = "9",
    pages = "1935",
    year = "2025"
}

@article{DeVuyst:2025ezt,
    author = "De Vuyst, Julian and Hoehn, Philipp A. and Tsobanjan, Artur",
    title = "{On the relation between perspective-neutral, algebraic, and effective quantum reference frames}",
    eprint = "2507.14131",
    archivePrefix = "arXiv",
    primaryClass = "quant-ph",
    year = "2025"
}

@article{Fedida:2025viy,
    author = "Fedida, Samuel and Glowacki, Jan",
    title = "{Foundations of relational quantum field theory I: Scalars}",
    eprint = "2507.21601",
    archivePrefix = "arXiv",
    primaryClass = "quant-ph",
    year = "2025"
}

@article{Fiore:2025oks,
    author = "Fiore, Gaetano and Lizzi, Fedele",
    title = "{Mixed states for reference frames transformations}",
    eprint = "2507.05758",
    archivePrefix = "arXiv",
    primaryClass = "quant-ph",
    year = "2025"
}

\end{document}